\documentclass{article}
\usepackage{graphicx} 

\usepackage{setspace}

\usepackage{amsmath}
\usepackage{amssymb}
\usepackage{amsthm}
\usepackage{mathtools}
\usepackage{empheq}
\usepackage{enumitem}
\usepackage{titlesec}
\usepackage{graphicx}
\usepackage{caption}
\usepackage{subcaption}
\usepackage{diagbox}
\usepackage{hyperref}

\usepackage[T1]{fontenc}
\usepackage[utf8]{inputenc}
\usepackage{comment}
\usepackage{indentfirst}
\usepackage[top=1.25in,left=1.25in,right=1.25in]{geometry}

\numberwithin{equation}{section}

\title{\Large{\uppercase{\bf Regularity and separation for Sierpi\'{n}ski products of graphs}}}
\author{\Large{Riccardo W. Maffucci}}
\date{}
\newcommand{\Addresses}{  
		R.W.~Maffucci, \textsc{Dipartimento di Matematica, Universit\`a di Torino\\\indent Via Carlo Alberto 10, Turin 10123, Italy}\par\nopagebreak\vspace{-0.35cm}
		\textit{E-mail address}, R.W.~Maffucci: \href{mailto:riccardowm@hotmail.com}{\texttt{riccardowm@hotmail.com}}
  }

\def\calP{\mathcal{P}}
\def\calR{\mathcal{R}}
\def\calS{\mathcal{S}}
\def\calT{\mathcal{T}}
\def\ca{\mathcal{A}}
\def\cb{\mathcal{B}}

\def\cg{\mathcal{G}}

\def\q{\hspace{0.04cm}\square\hspace{0.04cm}}
\def\c{\circ}

\def\t{\otimes_f}
\def\s{\otimes}
\DeclareMathOperator{\ii}{im}

\newtheorem{thm}{Theorem}[section]
\newtheorem{lemma}[thm]{Lemma}
\newtheorem{prop}[thm]{Proposition}

\newtheorem{cor}[thm]{Corollary}
\newtheorem{defin}[thm]{Definition}

\newtheorem{rem}[thm]{Remark}

\begin{document}
\titleformat{\section}
  {\Large\scshape}{\thesection}{1em}{}
\titleformat{\subsection}
  {\large\scshape}{\thesubsection}{1em}{}
\maketitle
\Addresses

\begin{abstract}
The Sierpi\'{n}ski product of graphs generalises the vast and relevant class of Sierpi\'{n}ski-type graphs, and is also related to the classic lexicographic product of graphs. Our first main results are necessary and sufficient conditions for the higher connectivity of Sierpi\'{n}ski products. Among other applications, we characterise the polyhedral ($3$-connected and planar) Sierpi\'{n}ski products of polyhedra.

Our other main result is the complete classification of the regular polyhedral Sierpi\'{n}ski products, and more generally of the regular, connected, planar Sierpi\'{n}ski products. 

To prove this classification, we introduce and study the intriguing class of planar graphs where each vertex may be assigned a colour in such a way that each vertex has neighbours of the same set of colours and in the same cyclic order around the vertex.

We also completely classify the planar lexicographic products.
\end{abstract}

\tableofcontents

\section{Introduction}
\subsection{Graph products and polyhedra}
We are interested in polyhedra that are products of graphs. The graphs we consider are simple and finite. By polyhedra we mean planar, $3$-connected graphs ($3$-polytopes), since these are precisely the $1$-skeleta (wireframes) of polyhedral solids, as established by the Rademacher-Steinitz Theorem \cite{radste}.

Polyhedral graphs are hence a special type of planar graphs. The planar is the most studied class of graphs. Motivation to study planar, $3$-connected graphs comes from their neat properties, including: they have a unique embedding in the sphere \cite{whit32} (hence also in the plane, up to choosing the external region); the planar dual is not only always unique (as it is independent of the embedding) and always a simple graph, but also always a polyhedron; every face (region) is delimited by a polygon (cycle), and moreover every pair of faces intersects either in the empty set, or at a vertex, or at an edge; all polyhedra may be constructed recursively starting from the pyramids (wheel graphs) and performing only two operations, namely adding an edge and expanding a vertex \cite{tutt61}; equivalently, they may be constructed recursively starting from the pyramids, adding edges and taking duals \cite{tutt61}; the class of maximal planar graphs is the subclass of polyhedra where all faces are triangular (triangulations of the sphere/plane).

The three fundamental and most studied graph products are the Kronecker (also called direct or tensor), Cartesian, and strong \cite[\S 4.1]{haimkl}. Which polyhedra are Kronecker products was established recently in \cite{maffucci2024classification}, and investigated further in \cite{de2024cancellation}. Polyhedral Cartesian \cite{behzad1969topological,maffucci2024classification} and strong \cite{jha1993note,maffucci2024classification} products have also been classified. We begin by considering the graph product that is considered to be the next most fundamental, namely the lexicographic product \cite[p.\ 44]{haimkl}. We classify the planar (and also the polyhedral) lexicographic products.


The main focus of this paper is a product closely related to the lexicographic, called the Sierpi\'{n}ski product. We investigate which polyhedra are Sierpi\'{n}ski products of graphs. This turns out to be a much richer question.

Given the graphs $A=(V(A),E(A))$ and $B=(V(B),E(B))$, we would like to define a graph product $A\star B$ that satisfies the following natural properties:
\begin{itemize}
\item
$\star$ is associative;
\item
$V(A\star B)=V(A)\times V(B)$;
\item
at least one of the projections $V(A\star B)\to V(A)$ and $V(A\star B)\to V(B)$ is a weak homomorphism, in the following sense.
\end{itemize}
We call the map $f:V(A)\to V(B)$ a graph homomorphism if it preserves adjacencies i.e., for every $u,v\in V(A)$ we have
\[uv\in E(A) \Rightarrow f(u)f(v)\in E(B).\]
Accordingly, an isomorphism is a bijective homomorphism. We call $f:V(A)\to V(B)$ a weak graph homomorphism if for every $u,v\in V(A)$ we have
\[uv\in E(A) \Rightarrow f(u)f(v)\in E(B) \text{ or } f(u)=f(v).\]

It turns out that there are essentially only four graph products satisfying the above three conditions, namely the Kronecker, Cartesian, strong, and lexicographic \cite[p.\ 44]{haimkl} (the word 'essentially' here means up to ignoring certain products that only inherit the adjacency structure of one factor \cite[p. 44]{haimkl}).

We will now write the definition of lexicographic product $\c$, and compare it with the more well-known Cartesian product $\q$. We write for the vertex sets
\[V(A\c B)=V(A\q B)=V(A)\times V(B),\]
and for the edge sets
\begin{equation*}
	E(A\c B)=\{(a_1,b_1)(a_{2},b_2) : (a_1=a_2 \text{ and } b_1b_2\in E(B)) \text{ or } a_1a_2\in E(A)\}
\end{equation*}
and
\begin{equation*}
	E(A\q B)=\{(a_1,b_1)(a_2,b_2) : (a_1=a_2 \text{ and } b_1b_2\in E(B)) \text{ or } (a_1a_2\in E(A) \text{ and } b_1=b_2)\}.
\end{equation*}

It follows that $A\q B$ is a spanning subgraph of $A\c B$. In particular, if $A\c B$ is planar then so is $A\q B$.

The Kronecker, Cartesian, and strong products are commutative, while the lexicographic product is not \cite[p.\ 43]{haimkl}. In the Kronecker, Cartesian, and strong products, both projections are weak homomorphisms, whereas in the lexicographic product, only one projection is a weak homomorphism.

One may construct $A\c B$ by starting with the graphs
\begin{equation}
\label{eq:aB}
\{aB : a\in V(A)\},
\end{equation}
where $aB\simeq B$ via the isomorphism
\[(a,b)\leftrightarrow b, \quad b\in V(B),\]
and then adding the edges
\[\{(a_1,b_1)(a_{2},b_2) : a_1a_2\in E(A)\}.\]

\subsection{Sierpi\'{n}ski product}
The main focus of this paper is the Sierpi\'{n}ski product of graphs, introduced in \cite{kpzz19} as a natural generalisation of the family of Sierpi\'{n}ski-type graphs. It was further investigated in \cite{henning2024resolvability,tian2024general}. The terminology 'Sierpi\'{n}ski-type graph' was introduced in \cite{hinz2017survey} to encompass classes such as Hanoi graphs, Sierpi\'{n}ski graphs, Sierpi\'{n}ski triangle graphs, and generalisations. These have recently attracted a considerable amount of attention \cite{klavvzar1997graphs,hinz2002planarity,kaimanovich2003random,klavvzar2005crossing,cristea2013distances,brevsar2018packing,fang2022topological}. The fact that the Tower of Hanoi problem with four or more pegs is difficult is not surprising from a graph-theoretical point of view seeing as, outside of a few small special cases, the corresponding Hanoi graphs are non-planar, whereas the Hanoi graphs for three or fewer pegs are all planar \cite{hinz2002planarity}.

To define the Sierpi\'{n}ski product of the graphs $A,B$ with respect to the fixed map
\[f:V(A)\to V(B),\]
one writes
\[V(A\t B)=V(A)\times V(B)\]
and
\begin{equation*}
	E(A\t B)=\{(a_1,b_1)(a_{2},b_2) : a_1=a_2 \text{ and } b_1b_2\in E(B)\}\cup\{(a_1,f(a_2))(a_2,f(a_1)) : a_1a_2\in E(A)\}.
\end{equation*}
Note that for every $f:V(A)\to V(B)$, the product $A\t B$ is a (spanning) subgraph of $A\c B$. As in the case of the lexicographic product, the Sierpi\'{n}ski product is non-commutative \cite[\S 2]{kpzz19}. One may define the Sierpi\'{n}ski product with multiple factors, though the product in non-associative \cite[\S 4]{kpzz19}. Similarly to the lexicographic product, one may construct $A\t B$ by starting with the graphs $aB$ \eqref{eq:aB}, and then adding the edges
\[\{(a_1,f(a_2))(a_2,f(a_1)) : a_1a_2\in E(A)\}.\]

In the special case where $V(A)\subseteq V(B)$ and $f$ is the identity function on its domain, we write $A\t B$ simply as
\[A\s B.\]
One has accordingly
\begin{equation*}
	E(A\s B)=\{(a_1,b_1)(a_{2},b_2) : a_1=a_2 \text{ and } b_1b_2\in E(B)\}\cup\{(a_1,a_2)(a_2,a_1) : a_1a_2\in E(A)\}.
\end{equation*}

We say that $f$ is a \textbf{locally injective} function if whenever $a_1\neq a_2$ are distinct neighbours of the same vertex, then one has $f(a_1)\neq f(a_2)$.

\paragraph{Basic definitions and terminology.} We say that a graph $G$ is $k$-connected if $G$ has more than $k$ vertices, and however one removes fewer than $k$ vertices, the result is a connected graph. We say that $G$ is of connectivity $k$ if $k$ is the maximum integer such that $G$ is $k$-connected.
\\
In a graph $G$, a subset $W\subset V(G)$ is a cut-set in $G$ if removing all elements of $W$ from $G$ increases the number of connected components. If $W$ has cardinality $k$ we call it a $k$-cut. If $W=\{w\}$ then we say that $w$ is a separating vertex. Note that if $W$ is a cut-set, then there exist $x,y\in V(G)$ such that every $xy$-path in $G$ contains at least one element of $W$. We say that $W$ is a minimal cut-set if it is a cut-set but every proper subset of $W$ is not a cut-set.
\\
A plane graph is a planar graph considered together with a planar embedding. In case of polyhedra, the embedding is unique up to choosing the external region. An outerplanar graph is a plane graph where all vertices appear around the boundary of the same region.
\\
When we consider the neighbours of a vertex in cyclic order around the vertex, this may be clockwise or counterclockwise.
\\
Vertices of graph products $A\star B$ will be written as $(a,b)$ where $a\in V(A)$ and $b\in V(B)$, or simply as $ab$ when there is no possibility of confusion.

\subsection{Main results}
The product $A\t B$ is connected if and only if both of $A,B$ are connected \cite[Proposition 2.12]{kpzz19}. Our first main results concern the \textbf{higher connectivity of Sierpi\'{n}ski products}. For $2$-connectivity, we have established the following necessary and sufficient condition.

\begin{thm}
	\label{thm:sep}
	Let $A,B$ be connected, non-trivial graphs, and $f:V(A)\to V(B)$ a function. Then
	\[(a,b)\in V(A\t B)\]
	is a separating vertex if and only if one of the following conditions is verified:
	\begin{enumerate}[label=(\roman*)]
		\item
		\label{eq:sepi}
		$a$ is a separating vertex of $A$, and moreover there exists a partition
		\[H_1,H_2,\dots,H_n, \ n\geq 2 \quad\text{ of } \ V(A-a)\]
		such that if $h_i\in H_i$ and $h',h_i$ are in the same connected component of $A-a$, then $h'\in H_i$, and such that if $h_i\in H_i$ and $h_j\in H_j$ are neighbours of $a$, where $i\neq j$, then the vertex $b$ lies on every
		\[f(h_i)f(h_j)\text{-path}\]
		in $B$;
		\item
		\label{eq:sepii}
		$b$ is a separating vertex of $B$, and moreover there exists a connected component $J$ of $B-b$ disjoint from
		\[\{f(a') : aa'\in E(A)\};\]
		\item
		\label{eq:sepiii}
		the image under $f$ of every neighbour of $a$ is $b$.
	\end{enumerate}
\end{thm}
Theorem \ref{thm:sep} will be proven in Section \ref{sec:high}.

For even higher connectivity, we have found a sufficient condition.

\begin{thm}
	\label{thm:kc}
	Let $A,B$ be $k$-connected graphs with $k\geq 1$, and $f:V(A)\to V(B)$ a function such that for every $a\in V(A)$,
	\[|\{f(a') : aa'\in E(A)\}|\geq k.\]
	Then $A\t B$ is $k$-connected.
\end{thm}
Theorem \ref{thm:kc} will be proven in Section \ref{sec:high}.

Our next main result characterises the polyhedral Sierpi\'{n}ski products of polyhedra. It combines Theorem \ref{thm:kc} with the characterisation of planar Sierpi\'{n}ski products of \cite[Theorem 2.16]{kpzz19}.

\begin{thm}
	\label{thm:siepol}
	Let $A,B$ be polyhedra and $f:V(A)\to V(B)$ a locally injective function. Then $A\t B$ is a polyhedron if and only if for every $a\in V(A)$ there exists a face $\calR_a$ of $B$ containing the vertices
	\begin{equation}	
		\label{eq:images}
		f(u_1),f(u_2),\dots,f(u_{\deg_A(a)}),
	\end{equation}
	in this cyclic order around the contour of $\calR_a$, where
	\[u_1,u_2,\dots,u_{\deg_A(a)}\]
	are the neighbours of $a$ in this cyclic order in the planar embedding of $A$, and where at least three of \eqref{eq:images} are distinct.
\end{thm}
Theorem \ref{thm:siepol} will be proven in Section \ref{sec:high}.

It is possible for $A\t B$ to be a polyhedron even when neither factor is a polyhedron (e.g., Figure \ref{fig:sp}). The class of polyhedral Sierpi\'{n}ski products, or \textbf{Sierpi\'{n}ski polyhedra}, appears to be quite vast.
\begin{figure}[ht]
	\centering
	\begin{subfigure}{0.47\textwidth}
		\centering
		\includegraphics[width=3.25cm]{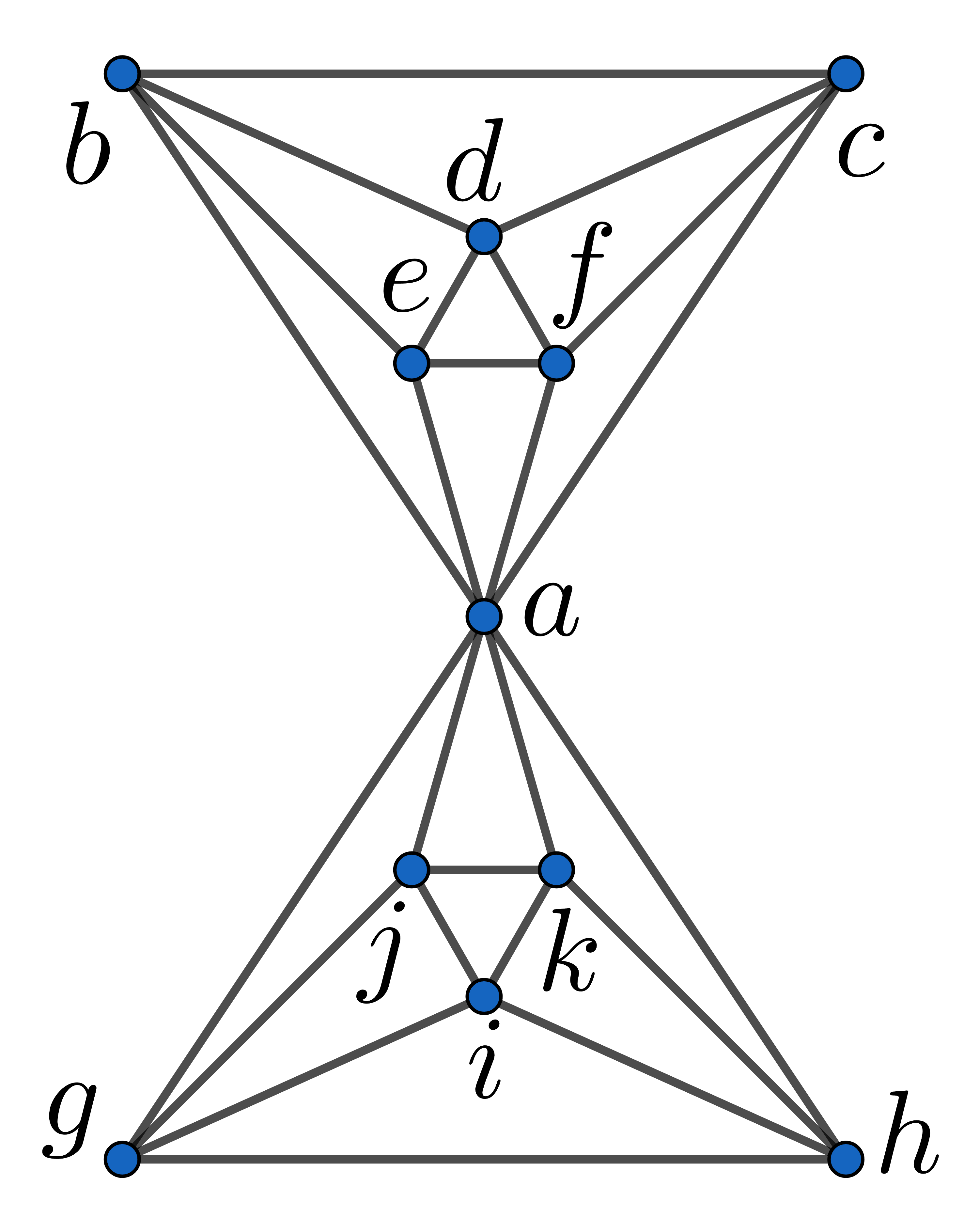}
		\caption{$A$.}
		\label{fig:sp1}
	\end{subfigure}
	\hspace{-1.5cm}
	\begin{subfigure}{0.47\textwidth}
		\centering
		\includegraphics[width=3.cm]{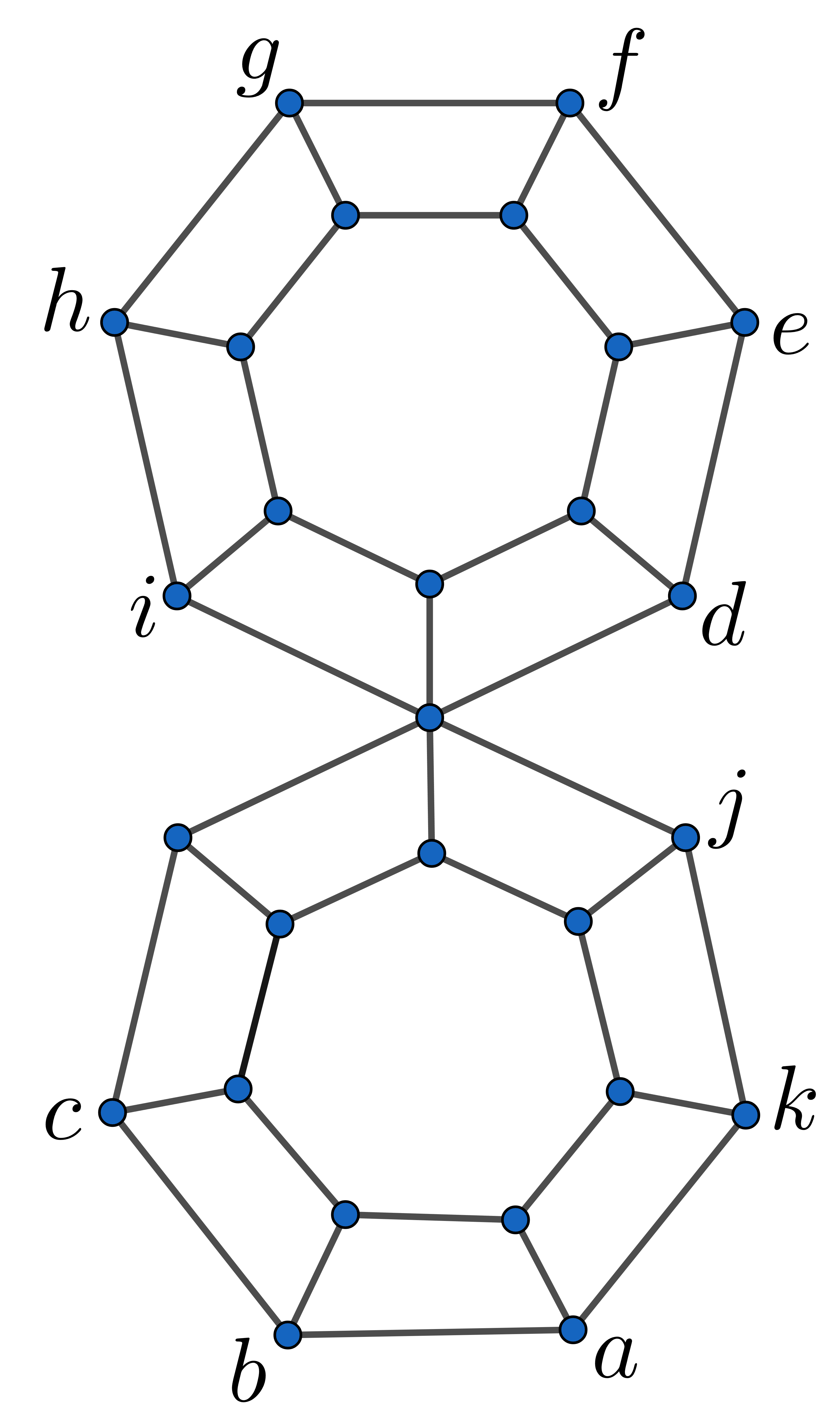}
		\caption{$B$.}
		\label{fig:sp2}
	\end{subfigure}
	\caption{$A,B$ are of connectivity $1$, while the Sierpi\'{n}ski product $A\s B$ is a polyhedron.}
	\label{fig:sp}
\end{figure}

A notable subclass of Sierpi\'{n}ski polyhedra are those that are also regular graphs. Polyhedra that are regular graphs are called vertex-regular polyhedra. Their duals are the face-regular polyhedra. For instance, the maximal planar graphs are the $3$-face-regular polyhedra. Regular polyhedra were already of interest to the ancient Greeks (e.g., Platonic solids). For recent results on regular polyhedra see e.g., \cite{brin05,hasheminezhad2011recursive,de2024cancellation,maffucci2025classification}.

We now focus on the \textbf{regular Sierpi\'{n}ski polyhedra}, and more generally on the regular, connected, planar Sierpi\'{n}ski products. We have found a complete classification for these classes. The statement requires a few definitions and the proof involves several considerations. The first factor is always a member of a specific class $\ca(n_1,n_2,\dots,n_k)$ of regular, connected, planar graphs that we will now define. This class is of significant interest in its own right. We have dedicated Section \ref{sec:A} to the theory of $\ca(n_1,n_2,\dots,n_k)$, focusing on its fine properties, subclasses, and partial classification.

\begin{defin}
	\label{def:A}
	Let $1\leq k\leq 5$ and
	\begin{equation}
		\label{eq:colord}
		n_1\geq n_2\geq\dots\geq n_k\geq 1
	\end{equation}
	be integers. We denote by
	\[\ca(n_1,n_2,\dots,n_k)\]
	the class of connected, planar, $n_1+n_2+\dots+n_k$-regular graphs $A$ satisfying the following condition. There exists a planar embedding of $A$ such that each vertex may be assigned one of five colours labelled by
	\begin{equation}
	\label{eq:col}
	c_1=\text{`red'},\quad c_2=\text{`blue'},\quad c_3=\text{`green'},\quad c_4=\text{`black'},\quad c_5=\text{`yellow'},
	\end{equation}
	in such a way that for every $1\leq i\leq k$, each $a\in V(A)$ has exactly $n_i$ neighbours of the colour $c_i$.
\end{defin}

Figure \ref{fig:A} depicts elements of $\ca(1,1,1)$, $\ca(2,1,1)$, and $\ca(2,2)$. Note that if in Figure \ref{fig:A211} we replace all the labels 3 with 2, we obtain a member of $\ca(2,2)$. On the other hand, the graph in Figure \ref{fig:A22} has no labeling that would allow us to obtain a member of $\ca(2,1,1)$. This remark will be generalised in Section \ref{sec:A}.
	\begin{figure}[ht]
		\centering
		\begin{subfigure}{0.32\textwidth}
			\centering
			\includegraphics[width=3.cm]{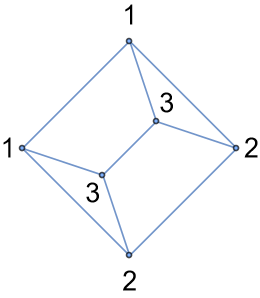}
			\caption{A member of $\ca(1,1,1)$.}
			\label{fig:A111}
		\end{subfigure}
		\hfill
		\begin{subfigure}{0.32\textwidth}
			\centering
			\includegraphics[width=3.75cm]{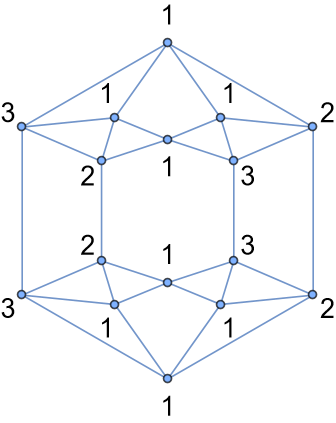}
			\caption{A member of $\ca(2,1,1)$.}
			\label{fig:A211}
		\end{subfigure}
		\hfill
		\begin{subfigure}{0.32\textwidth}
			\centering
			\includegraphics[width=3.75cm]{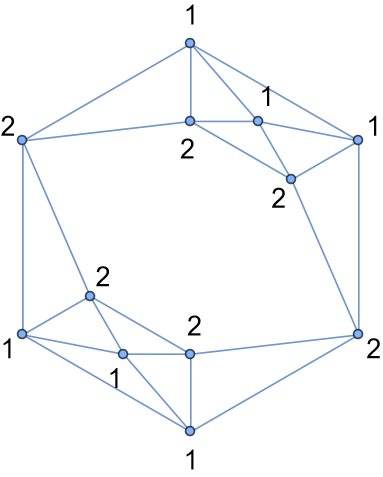}
			\caption{A member of $\ca(2,2)$.}
			\label{fig:A22}
		\end{subfigure}
		\caption{Labels represent colours of vertices.}
		\label{fig:A}
	\end{figure}

The second factor of a polyhedral Sierpi\'{n}ski product always belongs to the following class of graphs.

\begin{defin}
	\label{def:B}
	Let $2\leq r\leq 5$, $1\leq k\leq 3$, and $n_1\geq n_2\geq\dots\geq n_k\geq 1$. We denote by
	\[\cb_r(n_1,n_2,\dots,n_k)\]
	the class of connected planar graphs $B$ with degree sequence
	\[\sigma(B): r,r,\dots,r,r-n_k,\dots,r-n_2,r-n_1,\]
	where for $1\leq i\leq k$ the vertex of degree $r-n_i$ is assigned the colour $c_i$ \eqref{eq:col}, and admitting a planar embedding where all the vertices of non-maximal degree lie on the same region.
\end{defin}

Note that $K_2\in\cb_3(2,2)$ and $K_3\in\cb_3(1,1,1)$.

\begin{defin}
	\label{def:col}
	Let $A\in\ca(n_1,n_2,\dots,n_k)$, $B\in\cb_r(n_1,n_2,\dots,n_k)$, and \[f:V(A)\to V(B).\]	
	We say that $f$ preserves colours if it maps each vertex of $A$ to the vertex of $B$ of the same colour.
\end{defin}

To state our classification theorem, we need a few further definitions.

\begin{defin}
	\label{def:hash}
	Let $\ca_\#(2,1,1)$ be the subclass of graphs $A\in\ca(2,1,1)$ such that for every $2$-cut
	\[\{a_1,a_2\}\]
	the graph $A-a_1-a_2$ has exactly two connected components, each containing two neighbours of $a_i$, exactly one of them being a red vertex, for at least one of $i=1,2$.
\end{defin}
Note that $3$-connected members of $\ca(2,1,1)$ such as the graph in Figure \ref{fig:A211} clearly belong to $\ca_\#(2,1,1)$.

\begin{defin}
	\label{def:b*}
	Let $\cb_r^*(n_1,n_2,n_3)$ be the subclass of graphs $B\in\cb_r(n_1,n_2,n_3)$ such that if $W$ is an $h$-cut, $h\in\{1,2\}$, then each connected component $J$ of $B-W$ verifies
	\begin{equation}
		\label{eq:hsum}
		\sum_{v\in V(J)}(r-\deg_B(v))\geq 3-h.
	\end{equation}
\end{defin}

We are ready to state our complete classification for the class of regular polyhedral Sierpi\'{n}ski products.

\begin{thm}
	\label{thm:class}
	Let $A,B$ be non-trivial graphs and $f: V(A)\to V(B)$ a function. Then $A\t B$ is an $r$-regular polyhedron if and only if we are in one of the six scenarios in Table \ref{t:1}, and $f$ preserves colours.
	\begin{table}[h!]
		\centering
		\Large
		$\begin{array}{|l|ll|l|c|}
			\hline r&A&&B&\text{scenario}\\
			\hline 3&\ca(2,2),&3\text{-connected}&K_2&1\\
			\hline 3&\ca(1,1,1),&3\text{-connected}&\cb_3^*(1,1,1)&2\\
			\hline 3&\ca_\#(2,1,1),&2\text{-connected}&\cb_3^*(2,1,1)&3\\
			\hline 4&\ca_\#(2,1,1),&2\text{-connected}&\cb_4^*(2,1,1)&4\\
			\hline 5&\ca(1,1,1),&3\text{-connected}&\cb_5^*(1,1,1)&5\\
			\hline 5&\ca_\#(2,1,1),&2\text{-connected}&\cb_5^*(2,1,1)&6\\
			\hline
		\end{array}$
		\caption{The six scenarios for $A,B,f$ such that $A\t B$ is a regular polyhedron.}
		\label{t:1}
	\end{table}
\end{thm}
Theorem \ref{thm:class} will be proven in Section \ref{sec:class} also applying the results of Sections \ref{sec:high} and \ref{sec:A}. Constructions for $\cb_r^*(1,1,1)$ and $\cb_r^*(2,1,1)$ may be found in Appendix \ref{sec:appa}.

Note that in scenarios 2 and 5, $f$ is locally injective, since in $A$ every vertex has exactly one neighbour of each colour $c_1,c_2,c_3$. On the other hand, in the remaining scenarios $f$ is not locally injective.

\paragraph{Plan of the paper.} The rest of this paper is organised as follows. In Section \ref{sec:lexi} we will classify the planar lexicographic products (and polyhedral lexicographic products). In Section \ref{sec:bas} we will present a few preliminary results on Sierpi\'{n}ski polyhedra, and classify the special polyhedral products of type $A\s A$.
\\
The main work is contained in Sections \ref{sec:high}, \ref{sec:A}, and \ref{sec:class}. In Section \ref{sec:high}, we will focus on the higher connectivity of Sierpi\'{n}ski products, and prove Theorems \ref{thm:sep}, \ref{thm:kc}, and \ref{thm:siepol}. Section \ref{sec:A} is dedicated to the fine properties of the graphs $\ca(n_1,n_2,\dots,n_k)$ of Definition \ref{def:A}. In Section \ref{sec:class}, we will complete the proof of Theorem \ref{thm:class}.
\\
Section \ref{sec:fam} characterises certain interesting classes of Sierpi\'{n}ski polyhedra (dropping the assumption of regularity). In Appendix \ref{sec:appa}, we deal with further properties and constructions for the classes $\cb_r^*(1,1,1)$ and $\cb_r^*(2,1,1)$ that appear in Theorem \ref{thm:class}. Appendix \ref{sec:appb} contains the classification of the wider class of regular, planar Sierpi\'{n}ski products.

\paragraph{Related results.}
The characterisation of polyhedral Kronecker products was established recently in \cite{maffucci2024classification}. One factor is always $K_2$. When the second factor is a planar graph, \cite[Theorems 1.3 and 1.4]{maffucci2024classification} give a precise description of the four types of polyhedral products, according to the number of odd faces and how they intersect. When the second factor is non-planar, \cite[Theorem 1.7]{maffucci2024classification} prescribes how to construct the corresponding products. Polyhedral Kronecker products were further investigated in \cite{de2024cancellation}. Remarkably, \cite[Theorem 1.10]{de2024cancellation} solves the question of Kronecker cancellation in the special case of planar, $3$-connected product. More precisely, a polyhedron is a Kronecker product in at most one way. This is not true in general for simple graphs e.g., the Desargues graph is a Kronecker product in two distinct ways \cite{imrich2008multiple}. There are two families of polyhedral Cartesian products: the stacked prisms (products of a cycle and a simple path) and the prisms of outerplanar, $2$-connected graphs (that is to say, the products of an outerplanar, $2$ connected graph -- i.e., a polygon with some added diagonals -- and $K_2$) \cite{behzad1969topological}, \cite[Proposition 1.9]{maffucci2024classification}. There are only two polyhedral strong products \cite{jha1993note}, \cite[Proposition 1.10]{maffucci2024classification}.

\subsection{Acknowledgements}
Riccardo W. Maffucci was partially supported by Programme for Young Researchers `Rita Levi Montalcini' PGR21DPCWZ \textit{Discrete and Probabilistic Methods in Mathematics with Applications}, awarded to Riccardo W. Maffucci.

\section{Planar lexicographic products}
\label{sec:lexi}

In this section, we completely classify the planar lexicographic products and polyhedral lexicographic products.

\begin{prop}
Let $A\c B$ be a planar graph, with $B$ non-trivial and $A$ non-empty. Then either $B\simeq\overline{K_2}$ and every connected component of $A$ is either a triangle or a tree, or $B\simeq K_2$ and every connected component of $A$ is a tree.
\\
Furthermore, the only polyhedral lexicographic products are the tetrahedon $K_4\simeq K_2\c K_2$ and the octahedron $K_3\c\overline{K_2}$.
\end{prop}
\begin{proof}
Suppose that $A\c B$ is a planar graph, with $B$ non-trivial and $A$ non-empty. Let us check that $B$ contains exactly two vertices. By contradiction, let $b_1,b_2,b_3\in V(B)$. As $A$ is non-empty, we may take $a_1a_2\in E(A)$. Then the set of vertices
\[\{(a_i,b_j) : \quad i=1,2, \quad j=1,2,3\}\]
induces a copy of $K(3,3)$ in $A\c B$, contradiction. Thus indeed $B$ has exactly two vertices, hence it is isomorphic to either $\overline{K_2}$ or $K_2$.

We note that $A\c\overline{K_2}$ is a (spanning) subgraph of $A\c K_2$. We check that $K_3\c\overline{K_2}$ is the octahedron, while $K_3\c K_2$ is non-planar.

Next, let us show that for $n\geq 4$, the product $C_n\c\overline{K_2}$ is non-planar. Letting $V(\overline{K_2})=\{x,y\}$ and
\[C_n=\{u_1,u_2,\dots,u_n\}\]
as usual, we take the partition
\begin{equation}
\label{eq:part}
\{\{u_1x,u_1y,u_3x\}, \ \{u_2x,u_2y,u_4y\}\}.
\end{equation}
In $C_n\c\overline{K_2}$, both of $u_2x$ and $u_2y$ are adjacent to all of $u_1x,u_1y,u_3x$, and moreover $u_4y$ is adjacent to $u_3x$. We consider the paths
\[u_4y,u_5y,\dots,u_ny,u_1y\]
and 
\[u_4y,u_5x,\dots,u_nx,u_1x\]
to conclude that the partition \eqref{eq:part} determines a $K(3,3)$-minor in $C_n\c\overline{K_2}$. It follows that each cycle of $A$ is a triangle.

We now check that $A$ cannot contain as subgraph the graph of edges
\[\{ca,cb,cd,ab\},\]
as its lexicographic product with $\overline{K_2}$ is non-planar. Therefore, every connected component of $A$ is either a triangle or a tree.

We have already checked that $K_3\c\overline{K_2}$ is the octahedron, while $K_3\c K_2$ is non-planar. We now show by iterative construction that if $A$ is a tree, then $A\c B$ is planar for $B\in\{\overline{K_2},K_2\}$. The base case is a star of edges
\[\{ca_1,ca_2,\dots,ca_n\}\]
where the central vertex $c$ is taken as the root. For the products, we sketch the vertices
\begin{equation}
\label{eq:line}
a_1x,a_1y,a_2x,a_2y,\dots,a_nx,a_ny
\end{equation}
in this order along a vertical line, and $cx,cy$ on either side of this line. We add the edges from $cx$ to all of \eqref{eq:line}, and from $cy$ to all of \eqref{eq:line}. We have obtained a planar embedding of the lexicographic product of the star and $\overline{K_2}$. If instead $B$ is $K_2$, note that we can add the edges
\[\{(bx,by) \ : \ b\in V(A)\}\]
without breaking planarity, as in Figure \ref{fig:le1}. Continuing the iterative construction of $A$, for the inductive step we add the neighbours of $a_1$ distinct from $c$. Now $a_1$ is considered to be the root of a subtree, a star with centre $a_1$ and remaining vertices
\[b_1,b_2,\dots,b_{\deg_A(a_1)-1},\]
say. In the product $A\c B$, this time the vertices
\[b_1x,b_1y,b_2x,b_2y,\dots,b_{\deg_A(a_1)-1}x,b_{\deg_A(a_1)-1}y\]
are sketched in this order along a horizontal line between $ax,ay$. The resulting product $A\c B$ remains planar  (Figure \ref{fig:le2}). One continues in this fashion until all the vertices and edges of the tree $A$ have been added, thus $A\c B$ is planar. The proof of the first statement of the present proposition is complete.
	\begin{figure}[ht]
	\centering
	\begin{subfigure}{0.47\textwidth}
		\centering
		\includegraphics[width=4.75cm]{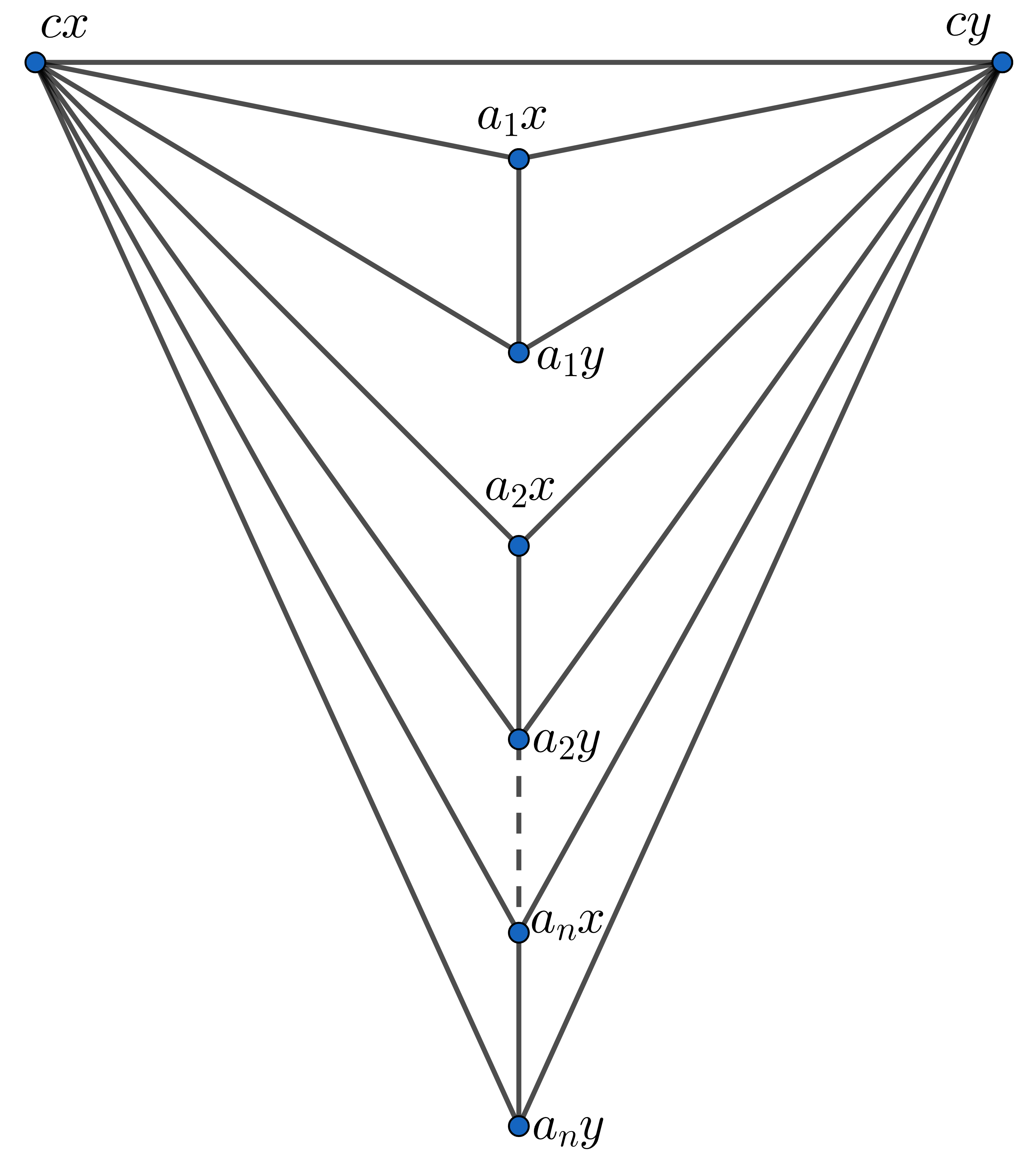}
		\caption{}
		\label{fig:le1}
	\end{subfigure}
	\hfill
	\begin{subfigure}{0.47\textwidth}
		\centering
		\includegraphics[width=4.75cm]{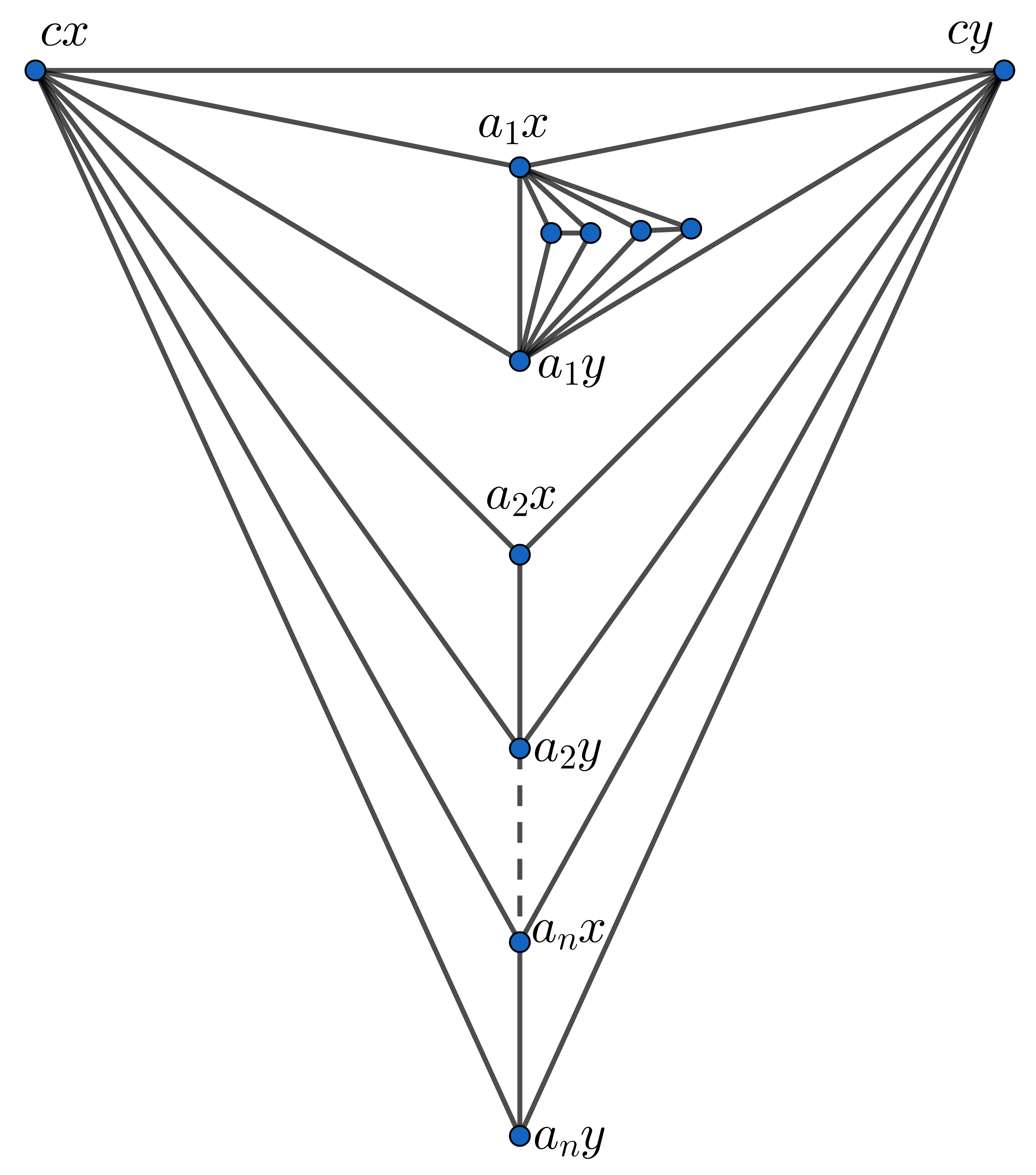}
		\caption{}
		\label{fig:le2}
	\end{subfigure}
	\caption{Construction of the planar lexicographic products.}
	\label{fig:le}
\end{figure}

We now assume that $A\c B$ is a polyhedron. By the first part of the present proposition, $B$ is either $\overline{K_2}$ or $K_2$. Moreover, since $A\c B$ is connected, by \cite[Corollary 5.14]{haimkl} $A$ is connected. By the first part of the present proposition, we conclude that $A$ is either a triangle or a tree. We already know that $K_3\c\overline{K_2}$ is the octahedron and that $K_3\c K_2$ is non-planar. Now let $A$ be a tree, $a\in V(A)$ of degree $1$, $c$ the only neighbour of $a$ in $A$, and $V(B)=\{x,y\}$. If $A$ has at least three vertices, then
\[\{cx,cy\}\]
is a $2$-cut in $A\c B$. It follows that the tetrahedron and octahedron are indeed the only polyhedral lexicographic products, as claimed.
\end{proof}

\section{Basics of polyhedral Sierpi\'{n}ski products}
\label{sec:bas}

In this section, we prove a few useful properties of polyhedral Sierpi\'{n}ski products, and classify those of type $A\s A$.

\begin{lemma}
	\label{le:deg3}
Let $A,B$ be non-trivial graphs and $f$ a function, such that $A\t B$ is a polyhedron. Then $A,B$ are planar and connected, $\delta(A)\geq 3$, and $|\ii(f)|\geq 2$. Moreover if $B=K_2$, then $\delta(A)\geq 4$.
\end{lemma}
\begin{proof}
Since $A\t B$ is connected, by \cite[Proposition 2.12]{kpzz19} both of $A,B$ are connected. By \cite[Lemma 2.5]{kpzz19}, $A\t B$ contains an $A$-minor and a copy of $B$, hence both of $A,B$ are planar.

Suppose that $|V(B)|\geq 3$. By contradiction, let $a\in V(A)$ be of degree at most $2$. If $(a,b)\in V(A\t B)$ is adjacent to a vertex outside of $aB$, then there exists $a'\in V(A)$ such that $b=f(a')$ and $aa'\in E(A)$. Therefore, at most two vertices of $aB$ are adjacent to vertices outside of $aB$. Hence there exists a $1$-cut or a $2$-cut in $A\t B$, contradiction. Similarly, if by contradiction we have $\ii(f)=\{b\}$, then since $B\not\simeq K_1$, for every $a\in V(A)$ the vertex $(a,b)$ is separating.

Now let $B=K_2$, with $V(B)=\{x,y\}$, and $a\in V(A)$. Since $A\t B$ is a polyhedron, its vertex $(a,x)$ has at least two neighbours other than $(a,y)$, say $(a_1,f(a))$ and $(a_2,f(a))$. Therefore, there exist at least two vertices $a_1,a_2\in V(A)$ adjacent to $a$ and such that $f(a_1)=f(a_2)=x$. Likewise, $(a,y)$ has at least two neighbours other than $(a,x)$, hence there exist at least two vertices $a_3,a_4\in V(A)$ adjacent to $a$ and such that $f(a_3)=f(a_4)=y$. Altogether, $a$ has at least four neighbours $a_1,a_2,a_3,a_4$ in $A$.
\end{proof}

\begin{lemma}
	\label{le:2conn}
	Let $A$ be a graph such that $A\s A$ is $2$-connected. Then $A$ is $2$-connected.
\end{lemma}
\begin{proof}
By contradiction, let $a\in V(A)$ be a separating vertex. Then there exist $u,v\in V(A)$ distinct from $a$ such that $a$ lies on every $uv$-path. Consider the vertices
\[(u,v), (v,u)\]
in the product. They are not adjacent, since $uv\not\in E(A)$. Let $\calP$ be a $(u,v)(v,u)$-path. By definition of Sierpi\'{n}ski product, the vertex following $(u,v)$ along $\calP$ is of the form $(u,v_1)$, with $vv_1\in E(A)$. The next vertex along the path is either of the form $(u,v_2)$, with $v_1v_2\in E(A)$, or if $uv_1\in E(A)$ possibly it is $(v_1,u)$. Continuing in this fashion, the first few vertices of $\calP$ are
\[(u,v), (u,v_1), (u,v_2), \dots, (u,v_n), (v_n,u),\]
with $n\geq 1$ and $uv_n\in E(A)$. We have determined in $A$ the path
\[u,v_n,v_{n-1},\dots,v_1,v.\]
Since $a$ belongs to every $uv$-path, then $a$ must be one of $v_1,v_2,\dots,v_n$. We have proven that in $A\s A$ the vertex $(u,a)$ lies on every $(u,v)(v,u)$-path. Hence $(u,a)$ is a separating vertex of $A\s A$, contradiction.
\end{proof}

The following basic lemma will be useful.
\begin{lemma}
	\label{le:out}
An outerplanar, $2$-connected graph $G$ is Hamiltonian. In fact, $G$ is simply a polygon with added diagonals that do not cross. It has at least two vertices of degree $2$.
\end{lemma}
\begin{proof}
Since $G$ is outerplanar, there exists a region containing all vertices. Since $G$ is $2$-connected, each region is bounded by a cycle. Thereby, there exists a cycle in $G$ containing all vertices i.e., $G$ is Hamiltonian. Any remaining edges in $G$ are diagonals of the polygon formed by the Hamiltonian cycle, and by planarity the diagonals cannot cross. It was shown in \cite[\S 3.2]{maffucci2024characterising} that such $G$ has at least two vertices of degree $2$. These are not part of any diagonals.
\end{proof}

We are ready to classify the polyhedral Sierpi\'{n}ski products of the form $A\s A$. In fact, there is a unique $2$-connected planar graph of the form $A\s A$.
\begin{prop}
	\label{prop:AA}
The graph $A\s A$ is $2$-connected and planar if and only if $A\simeq K_4$.
\end{prop}
\begin{proof}
Let $A\s A$ be $2$-connected and planar. By \cite[Theorem 2.16]{kpzz19}, every block of $A$ is either outerplanar or a copy of $K_4$. By Lemma \ref{le:2conn}, $A$ is $2$-connected, thus $A$ itself is either outerplanar or isomorphic to $K_4$. If $A$ is outerplanar, then by Lemma \ref{le:out}, $A$ contains at least two vertices of degree $2$. Invoking Lemma \ref{le:deg3}, we thus exclude the case of $A$ outerplanar, hence $A\simeq K_4$. We note that $K_4\s K_4$ (Figure \ref{fig:sie}) is $2$-connected and planar (in fact, it is a polyhedron), and the proof is complete.
\begin{figure}[ht]
\centering
\includegraphics[width=3.25cm]{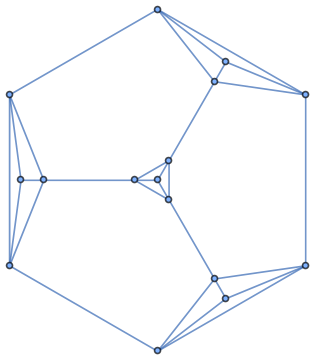}
\caption{$K_4\s K_4$.}
\label{fig:sie}
\end{figure}
\end{proof}

On the other hand, it is possible for $A,B$ to be both of connectivity $1$ and $A\t B$ a polyhedron, even when $V(A)\subseteq V(B)$ and $f:V(A)\to V(B)$ is the identity on its domain. For instance, we may construct $A$ as two octahedra sharing a vertex, and $B$ as two heptagonal prisms sharing a vertex, labeled as in Figure \ref{fig:sp}.

\section{Higher connectivity for Sierpi\'{n}ski products}
\label{sec:high}

\subsection{Proof of Theorems \ref{thm:sep}, \ref{thm:kc}, and \ref{thm:siepol}}

Recall that by \cite[Proposition 2.12]{kpzz19}, the product $A\t B$ is connected if and only if both of $A,B$ are connected. The goal of this section is to inspect higher connectivity for Sierpi\'{n}ski products.

We start with a useful lemma that allows us to construct a path between vertices in the product starting from a path in $A$.

\begin{lemma}
	\label{le:path}
Let $A,B$ be connected graphs, $a',a''\in V(A)$ and $b',b''\in V(B)$. Given a path in $A$ from $a'$ to $a''$
\begin{equation}
	\label{eq:u0}
	a'=u_0,u_1,u_2,\dots,u_{m-1},a''=u_m, \qquad m\geq 0,
\end{equation}
we can always find a path $P$ in $A\t B$ from $(a',b')$ to $(a'',b'')$ such that the first entries of the vertices of $P$ all belong to \eqref{eq:u0}.
\end{lemma}
\begin{proof}
Given \eqref{eq:u0}, we find in $B$ a $b_if(u_1)$-path, a $f(u_{m-1})b'$-path, and for each $\ell=1,\dots,m-1$, a $f(u_{\ell-1})f(u_{\ell+1})$-path, that exist by connectivity of $B$. This means we can find in $a'B$ a path $P_{0}$ from $(a',b')$ to $(a',f(u_1))$, in $a''B$ a path $P_{m}$ from $(a'',f(u_{m-1}))$ to $(a'',b'')$, and for each $\ell=1,\dots,m-1$ in $u_{\ell}B$ a path $P_{\ell}$ from $(u_{\ell},f(u_{\ell-1}))$ to $(u_{\ell},f(u_{\ell+1}))$. Finally we build $P$ by combining the following edges and paths,
\[P_0,(a',f(u_1))(u_1,f(a')),P_1,(u_1,f(u_2))(u_2,f(u_1)),P_2,\dots,P_{m-1},(u_{m-1},f(a''))(a'',f(u_{m-1})),P_m.\]
\end{proof}

We are ready to prove Theorem \ref{thm:sep}.

\begin{proof}[Proof of Theorem \ref{thm:sep}]
Assume that condition \ref{eq:sepi} holds. Take $h_1\in H_1$ and $h_2\in H_2$. Letting $b_1,b_2\in V(B)$, we will prove that $(a,b)$ lies on every
\[(h_1,b_1)(h_2,b_2)\text{-path}.\]
Let $\calP$ be such a path. We build a tuple $\calT$ with the first entries of the vertices of $\calP$ taken in order. By condition \ref{eq:sepi} and the definition of Sierpi\'{n}ski product, for $i=1,2$ if $h_i\in H_i$ and $(h_i,b_i)(h',b')\in E(A\t B)$, then either $h'=a$ or $h'\in H_i$. Therefore, in $\calT$ an element of $H_1$ may only be followed by $a$ or by another element of $H_1$. However, $\calP$ must contain also vertices with first entry in some $H_j$ with $j\neq 1$. An element of $H_j$ in $\calT$ may only be preceded by $a$ or another element of $H_j$. Hence in $\calT$ we will have some instance of $h_1'\in H_1$ followed by one or more instances of $a$ followed by some $h_j'\in H_j$. These correspond to a sub-path in $\calP$ of the form
\[(h_1',f(a)),(a,f(h_1')),\dots,(a,f(h_j')),(h_j',f(a)),\]
where the omitted vertices lie in $aB$. By condition \ref{eq:sepi}, $b$ belongs to every $f(h_1')f(h_j')$-path in $B$, hence by definition of Sierpi\'{n}ski product, the vertex $(a,b)$ appears somewhere along $\calP$, as claimed.

Now assume that condition \ref{eq:sepii} holds. Consider a vertex $(a,b')$ of the graph $aJ$. By assumption and by definition of Sierpi\'{n}ski product, $(a,b')$ can only be adjacent to $(a,b)$ or to another vertex of $aJ$. Letting $v\in V(A)$ with $v\neq a$, and taking any $x\in V(B)$, the vertex $(a,b)$ lies on every
\[(a,b')(v,x)\text{-path}\]
in $A\t B$.

Now assume that condition \ref{eq:sepiii} holds. Similarly to the previous case, take $x\in V(B)$ with $x\neq b$. Let $v\in V(A)$ with $v\neq a$, and $y\in V(B)$. Consider a
\[(a,x)(v,y)\text{-path}\]
in $A\t B$, call $(a',x')$ the first vertex with $a'\neq a$ lying on the path, and $(a,z)$ the one preceding it along the path. By definition of Sierpi\'{n}ski product, $a'a\in E(A)$ and $z=f(a')$. By assumption $f(a')=b$, thus $(a,b)$ lies on every $(a,x)(v,y)$-path, as required.

Vice versa, let $(a,b)$ be a separating vertex in $A\t B$, and
\[G_1,G_2,\dots,G_n, \qquad n\geq 2\]
the connected components of
\[A\t B-(a,b).\]
Each of these components contains a vertex adjacent to $(a,b)$.

For $1\leq i\leq n$ we define $H_i$ by taking the first entries of all elements of $G_i$, and removing $a$ if present. Such $H_i$'s are subsets of $V(A-a)$. 
We claim that
\begin{equation}
\label{eq:partH}
H_1,H_2,\dots,H_n
\end{equation}
is a partition of $V(A-a)$. Since all vertices of $A\t B$ where the first component is not $a$ appear in one of the $G_i$'s, then all elements of $V(A-a)$ clearly appear in at least one $H_i$. Hence to prove that \eqref{eq:partH} is a partition, it remains to show that if $i\neq j$ then
\[H_i\cap H_j=\emptyset.\]
By contradiction, there exists $v\in H_i\cap H_j$ with $i\neq j$. Then there exist $(v,b_i)\in V(G_i)$ and $(v,b_j)\in V(G_j)$, where $b_i,b_j$ are distinct vertices of $B$. Thus every
\[(v,b_i)(v,b_j)\text{-path}\]
contains $(a,b)$. However, this is impossible, since $vB\simeq B$ is connected, hence by definition of Sierpi\'{n}ski product there exists a $(v,b_i)(v,b_j)$-path contained entirely within $vB$. Thus \eqref{eq:partH} is indeed a partition of $V(A-a)$.

Now let $h_i\in H_i$, and $h'\in V(A)$ in the same component of $A-a$ as $h_i$. Then we have a $h_ih'$-path in $A$ that does not contain $a$. Letting $b_i,b'\in V(B)$, we build via Lemma \ref{le:path} a $(h_i,b_i)(h',b')$-path in $A\t B$ that does not contain $(a,b)$. Since $(h_i,b_i)\in V(G_i)$, we also have $(h',b')\in V(G_i)$, thus $h'\in H_i$ as in \ref{eq:sepi}.


Suppose for the moment that $a$ is a separating vertex of $A$. This condition is equivalent to the existence of two neighbours of $a$ in distinct connected components of $A-a$. By the preceding argument, we may fix these neighbours to be $h_i\in H_i$ and $h_j\in H_j$ with $i\neq j$. We build a path
\begin{equation}
	\label{eq:path12}
(h_i,f(a)),(a,f(h_i)),\dots,(a,f(h_j)),(h_j,f(a)),
\end{equation}
where the $(a,f(h_i))(a,f(h_j))$-subpath lies entirely within the connected graph $aB\simeq B$. Since $(h_i,f(a))\in V(G_i)$ and $(h_j,f(a))\in V(G_j)$, then the separating vertex $(a,b)$ lies on \eqref{eq:path12}. Now by definition of Sierpi\'{n}ski product, $b$ lies on every $f(h_i)f(h_j)$-path in $B$. We have obtained condition \ref{eq:sepi}.

Henceforth we may assume that $a$ is not a separating vertex of $A$. Thus if $H_1\neq\emptyset$, then we can fix $h_1\in H_1$ and deduce that, since $a$ is not a separating vertex, then every $h'\in V(A-a)$ is in the same component of $A-a$ as $h_1$, hence as shown above $h'$ belongs to $H_1$ as well. It follows that $H_2=H_3=\dots=H_n=\emptyset$. Then $G_1$ contains all vertices of $A\t B$ where the first component is not $a$, while $G_2,G_3,\dots,G_n$ only contain vertices where the first component is $a$.

Suppose that there exist $x,y\in V(B)$ such that $(a,x)\in V(G_i)$ and $(a,y)\in V(G_j)$, $i\neq j$. That is to say, $(a,b)$ lies on every $(a,x)(a,y)$-path in $A\t B$. In particular, $b$ is a separating vertex of $B$, and $x,y$ lie in distinct connected components of $B-b$. 
Letting $J_x$ and $J_y$ be the components of $B-b$ containing $x,y$ respectively, we claim that at least one of $J_x,J_y$ is disjoint from
\[\{f(a') : aa'\in E(A)\}.\]
By contradiction, let $J_x$ contain $f(a_x)$ and $J_y$ contain $f(a_y)$, where $a_x$ and $a_y$ are neighbours of $a$. Then we may find paths $\calP_x$ in $aJ_x$ from $(a,x)$ to $(a,f(a_x))$ and $\calP_y$ in $aJ_y$ from $(a,y)$ to $(a,f(a_y))$ that do not contain $(a,b)$. Since $a$ is not a separating vertex of $A$, we can find a path in $A$ from $f(a_x)$ to $f(a_y)$ that does not contain $a$. Via Lemma \ref{le:path}, we may build a path $\calP$ in $A\t B$ between $(a_x,f(a))$ and $(a_y,f(a))$, such that the first entry of each vertex of $\calP$ is always different from $a$. Thus $\calP$ does not contain $(a,b)$. We may then combine the following paths and edges,
\[\calP_x,(a,f(a_x))(a_x,f(a)),\calP,(a_y,f(a))(a,f(a_y)),\calP_y\]
to obtain a path between $(a,x)$ and $(a,y)$ that does not include $(a,b)$, contradiction. We have obtained the condition \ref{eq:sepii} with $J\in\{J_x,J_y\}$.

The only other possibility is that there exists $i\in\{1,2,\dots,n\}$ such that $(a,x)$ lies in $G_i$ for every $x\in V(B-b)$. We cannot have $i=1$, otherwise $G_2,G_3,\dots,G_n$ would contain no vertices, hence $n=1$ and $(a,b)$ could not be a separating vertex of $A\t B$. It follows that $i=n=2$. That is to say,
\[A\t B-(a,b)\]
has exactly two connected components $G_1,G_2$, satisfying
\[V(G_1)=\{(a',b') : a'\neq a\}\]
and
\[V(G_2)=\{(a,x) : x\in V(B-b)\}.\]
Hence the only vertex of $A\t B$ with first component equal to $a$ and adjacent to some $(a',b')$ where $aa'\in E(A)$ is $(a,b)$ itself. By definition of Sierpi\'{n}ski product, all neighbours $a'$ of $a$ satisfy $f(a')=b$. We have obtained condition \ref{eq:sepiii}, and the proof of the present theorem is complete.
\end{proof}

As an illustration of Theorem \ref{thm:sep}, the graphs in Figures \ref{fig:sepv1}, \ref{fig:sepv2}, and \ref{fig:sepv3} satisfy \ref{eq:sepi}, \ref{eq:sepii}, and \ref{eq:sepiii} respectively. 
	\begin{figure}[ht]
		\centering
		\begin{subfigure}{0.32\textwidth}
			\centering
			\includegraphics[width=4.25cm]{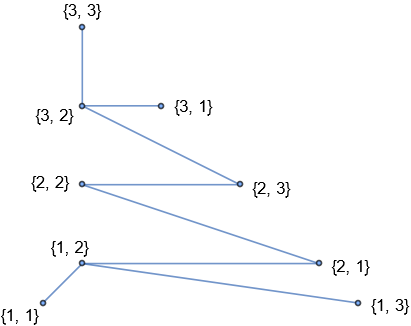}
			\caption{$E(A)=E(B)=\{12,23\}$ and $f$ is the identity.}
			\label{fig:sepv1}
		\end{subfigure}
		\hfill
		\begin{subfigure}{0.32\textwidth}
			\centering
			\includegraphics[width=4.75cm]{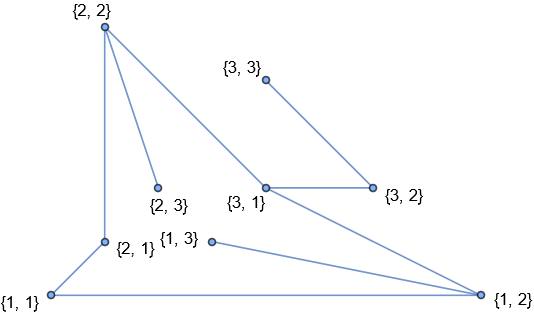}
			\caption{$A=K_3$, $E(B)=\{12,23\}$, and $f(1)=f(2)=1$, $f(3)=2$.}
			\label{fig:sepv2}
		\end{subfigure}
		\hfill
		\begin{subfigure}{0.32\textwidth}
			\centering
			\includegraphics[width=4.25cm]{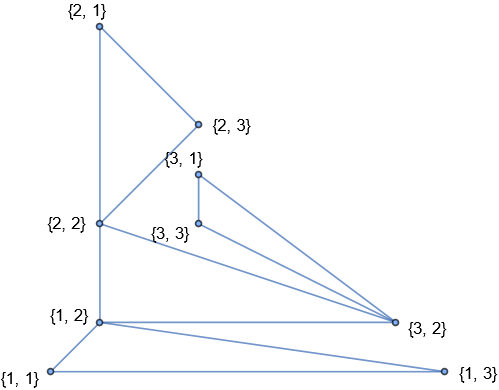}
			\caption{$A=B=K_3$ and $f(1)=f(2)=f(3)=2$.}
			\label{fig:sepv3}
		\end{subfigure}
		\caption{In these examples, $V(A)=V(B)=\{1,2,3\}$. In each case $(2,2)$ is a separating vertex in $A\t B$.}
		\label{fig:sepv}
	\end{figure}



We deduce the following useful consequence.
\begin{cor}
	\label{cor:imf}
	Let $A,B$ be graphs such that $A\t B$ is $3$-connected. Then $|im(f)|\geq 2$, and if $B\not\simeq K_2$, then $|im(f)|\geq 3$.
\end{cor}
\begin{proof}
	By contradiction, let $im(f)=\{b\}$. By Theorem \ref{thm:sep}, for every $a\in V(A)$ the vertex $(a,b)$ is a separating vertex of $A\t B$, contradiction.
	
	To prove the second statement, again by contradiction, let $|V(B)|\geq 3$ and \[im(f)=\{b_1,b_2\}.\]
	The graph $B-b_1-b_2$ contains at least one vertex, hence there exists a connected component $J$ of $B-b_1-b_2$ containing no element of $im(f)$. Let $b\in V(J)$ and consider the vertices
	\[(a,b), (a',b), \quad a\neq a'\in V(A).\]
	Since $im(f)=\{b_1,b_2\}$ and $a\neq a'$, every $(a,b)(a',b)$-path contains one of $(a,b_1),(a,b_2)$, hence $A\t B$ has a $2$-cut, contradiction.
\end{proof}

We now turn to proving Theorem \ref{thm:kc}.

\begin{prop}
	\label{prop:kc}
Let $A,B$ be connected, non-trivial graphs, and $f:V(A)\to V(B)$ a function. Suppose that
\[\{(a_1,b_1),(a_2,b_2),\dots,(a_k,b_k)\}, \qquad k\geq 1\]
is a minimal $k$-cut of $A\t B$. 
Then either
\[\{a_1,a_2,\dots,a_k\}\]
is a $k$-cut of $A$, or $a_1=a_2=\dots=a_k$ and 
\[\{b_1,b_2,\dots,b_k\}\]
is a $k$-cut of $B$, or $a_1=a_2=\dots=a_k=:a$ and for every neighbour $a'$ of $a$, we have
\[f(a')\in\{b_1,b_2,\dots,b_k\}.\]
\end{prop}
\begin{proof}
There exist $v,w\in V(A)$ and $x,y\in V(B)$ such that every
\[(v,x)(w,y)\text{-path}\]
contains one of
\begin{equation}
\label{eq:kcut}
\{(a_1,b_1),(a_2,b_2),\dots,(a_k,b_k)\}.
\end{equation}
Via Lemma \ref{le:path}, we start with a $vw$-path in $A$
\[v=u_0,u_1,u_2,\dots,u_{m-1},w=u_m, \qquad m\geq 0,\]
and build a $(v,x)(w,y)$-path $\calP$ in $A\t B$. Every vertex of $\calP$ has first component equal to one of $u_0,u_1,\dots,u_m$. One of \eqref{eq:kcut} lies on $\calP$. Suppose that
\[\{a_1,a_2,\dots,a_k\}\cap\{v,w\}=\emptyset.\]
Then one of $a_1,a_2,\dots,a_k$ belongs to $\{u_1,u_2,\dots,u_{m-1}\}$. Therefore, every $uv$-path in $A$ contains one of $a_1,a_2,\dots,a_k$, thus $\{a_1,a_2,\dots,a_k\}$ is a $k$-cut in $A$, and the proof is complete in this case.

The only other option w.l.o.g.\ is that, in Lemma \ref{le:path}, in the construction of the path $\calP_{0}$ in $vB$ from $(v,b')$ to $(v,f(u_1))$, we must always include one of \eqref{eq:kcut}. That is to say, $v\in\{a_1,a_2,\dots,a_k\}$.  W.l.o.g.\ we may fix $a_1=v$. Thus every $(a_1,x)(a_1,f(u_1))$-path in $a_1B$ contains an element of \eqref{eq:kcut}. Since \eqref{eq:kcut} is a minimal $k$-cut, we obtain
\[a_1=a_2=\dots=a_k=:a.\]
Thereby, every
\[(a,x)(a,f(u_1))\text{-path}\]
contains one of $(a,b_1)$, $(a,b_2)$, \dots, $(a,b_k)$. Recall that $(v,x)\neq (a,b_1), (a,b_2), \dots, (a,b_k)$, hence $x\neq b_1,b_2,\dots,b_k$. Suppose that $f(u_1)\neq b_1,b_2,\dots,b_k$. Then in $B$ every $xf(u_1)$-path contains one of $b_1,b_2,\dots,b_k$, i.e.\ $\{b_1,b_2,\dots,b_k\}$ is a $k$-cut in $B$, and the proof is complete in this case also.

It remains to consider the case where $a:=a_1=a_2=\dots=a_k=v$, and moreover for each choice of $u_1$ in the construction of $\calP$, we have
\[f(u_1)\in \{b_1,b_2,\dots,b_k\}.\]
Here every neighbour of $a$ is mapped to one of $b_1,b_2,\dots,b_k$ under $f$. The proof is complete.
\end{proof}

We are is a position to prove Theorem \ref{thm:kc}.

\begin{proof}[Proof of Theorem \ref{thm:kc}]
The case $k=1$ follows from \cite[Proposition 2.12]{kpzz19}. For $k=2$, since $A,B$ do not have separating vertices, then conditions \ref{eq:sepi} and \ref{eq:sepii} of Theorem \ref{thm:sep} do not hold. Moreover, by the assumption on $f$, condition \ref{eq:sepiii} of Theorem \ref{thm:sep} does not hold either. By Theorem \ref{thm:sep}, $A\t B$ does not have separating vertices. The proof for the case $k=2$ is complete. Now for every $k\geq 3$, invoking Proposition \ref{prop:kc} we conclude that $A\t B$ cannot have a $k-1$-cut, hence it is $k$-connected.
\end{proof}

We are now ready to prove Theorem \ref{thm:siepol}, characterising polyhedral Sierpi\'{n}ski products with polyhedral factors.

\begin{proof}[Proof of Theorem \ref{thm:siepol}]
$\Leftarrow$. By \cite[Theorem 2.16]{kpzz19}, $A\t B$ is a planar graph. Since for every $a\in V(A)$ at least three of \eqref{eq:images} are distinct, we may conclude by Theorem \ref{thm:kc} that $A\t B$ is $3$-connected.

$\Rightarrow$. By \cite[Theorem 2.16]{kpzz19}, for every $a\in V(A)$ there exists a face $\calR_a$ of $B$ containing the vertices
\begin{equation}	
	\label{eq:imagespf}
	f(u_1),f(u_2),\dots,f(u_{\deg_A(a)}),
\end{equation}
in this cyclic order around the contour of $\calR_a$, where
\[u_1,u_2,\dots,u_{\deg_A(a)}\]
are the neighbours of $a$ in this cyclic order in the planar embedding of $A$. Now by contradiction, suppose that only one or two of \eqref{eq:imagespf} are distinct. Similarly to the proof of Theorem \ref{thm:sep}, take $(a,x)\in V(A\t B)$, $x$ not belonging to \eqref{eq:imagespf}, and any $(v,y)\in V(A\t B)$ with $v\neq a$. By definition of Sierpi\'nski product, every $(a,x)(v,y)$-path contains one of \eqref{eq:imagespf}, so that $A\t B$ has a $2$-cut, contradiction.
\end{proof}

\subsection{Consequences and applications}
Given a cubic polyhedron $A$, the following corollary shows how to find $B$ and $f$ such that $A\t B$ is a polyhedron.

\begin{cor}
	Let $A$ be a cubic polyhedron, and $B$ the $|V(A)|$-gonal pyramid. Then for every bijection $f$ between $V(A)$ and the non-central vertices of $B$, the graph
	\[A\t B\]
	is a polyhedron.
\end{cor}
\begin{proof}
Since every vertex of $a$ has degree $3$, then the base of the pyramid $B$ contains $f(a_1)$, $f(a_2)$, and $f(a_3)$ in cyclic order, where $a_1,a_2,a_3$ are the neighbours of $a$. Moreover, $A,B$ are $3$-connected and $f$ is a bijection, hence by Theorem \ref{thm:siepol} $A\t B$ is a polyhedron.
\end{proof}

We may strengthen the above result to polyhedra where the vertices of degree four or higher are not too closely connected. Henceforth in the graph $G$ the quantity $d_G(u,v)$ denotes the distance between the vertices $u,v$.
\begin{cor}
	\label{cor:tree}
Let $A$ be a polyhedron and $B$ an $n$-gonal pyramid with $n\geq |V(A)|$. Let $A'$ be the graph with
\[V(A')=\{v\in V(A) : \deg_A(v)\geq 4\}\]
and
\[E(A')=\{vw : v,w\in V(A), d_A(v,w)\leq 2\}.\]
If $A'$ is a forest, then there exists an injection
\[f:V(A)\to V(B)\]
such that
\[A\t B\]
is a polyhedron.
\end{cor}
\begin{proof}
We build a list with the elements of $V(A)$ recursively in the following way. Given a tree $T$ of $A'$, we start by considering two adjacent vertices $a_1,a_2$. We list the neighbours $\neq a_2$ of $a_1$ in $A$ in counter-clockwise order
\begin{equation}
\label{eq:na1}
u_1,u_2,\dots,u_{\deg_A(a_1)}.
\end{equation}
By assumption there exists a labeling of the neighbours $\neq a_1$ of $a_2$ in $A$ in clockwise order
\begin{equation}
	\label{eq:na2}
	v_1,v_2,\dots,v_{\deg_A(a_2)}
\end{equation}
such that there exist
\[i_1<i_2<\dots<i_k, \quad j_1<j_2<\dots<j_k, \quad k\geq 0,\]
with
$u_{i_1}=v_{j_1}$, $u_{i_2}=v_{j_2}$, \dots, and $u_{i_k}=v_{j_k}$, where $k$ is the number of common neighbours of $a_1,a_2$ in $A$. The reader may refer to Figure \ref{fig:tree}. If one set of indices is in ascending order, then the other set is in ascending order as well due to planarity of $A$. We build a new list by combining \eqref{eq:na1} and \eqref{eq:na2},
\begin{align}
\label{eq:longlist}
\notag
u_1,u_2,\dots,u_{i_1-1},v_1,v_2,\dots,v_{j_1-1},u_{i_1},u_{i_1+1},u_{i_1+2},\dots,u_{i_2-1},v_{j_1+1},v_{j_1+2},\dots,v_{j_2-1},u_{i_2},\dots,u_{i_k},\\u_{i_k+1},\dots,u_{\deg_A(a_1)},v_{j_k+1},\dots,v_{\deg_A(a_2)}.
\end{align}
In \eqref{eq:longlist}, the neighbours of $a_1$ appear in cyclic order, and likewise the neighbours of $a_2$. We then consider a neighbour $a_3$ of $a_1$ or of $a_2$ in $A'$, and combine \eqref{eq:longlist} with the list of neighbours of $a_3$ in $A$ taken in cyclic order. Since $T$ is a tree, we may write the list of neighbours of $a_3$ in $A$ so that common vertices with the combined list of neighbours of $a_1$ or $a_2$ appear in the same order.
\begin{figure}[ht]
\centering
\includegraphics[width=7.5cm]{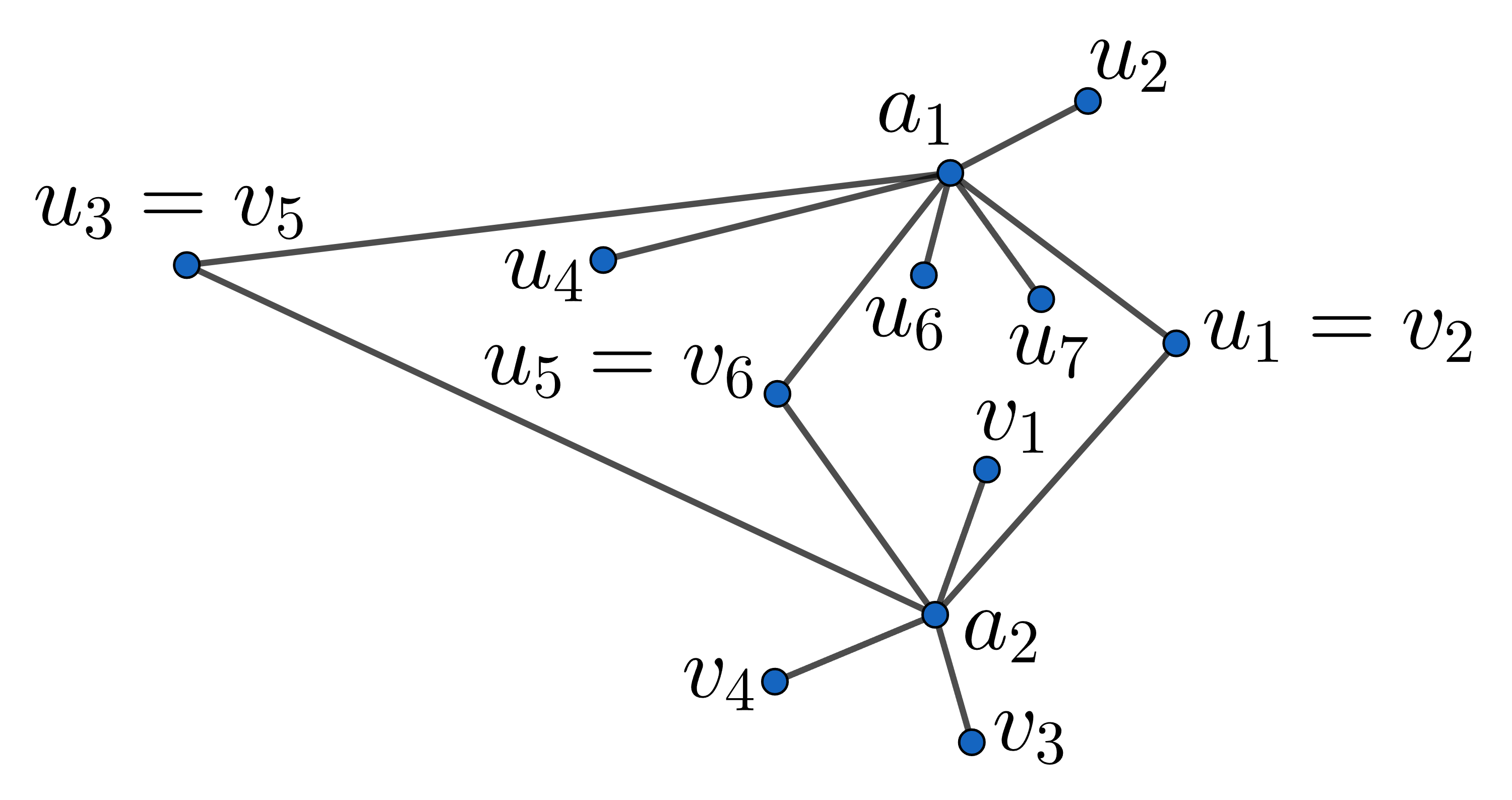}
\caption{Corollary \ref{cor:tree}}
\label{fig:tree}
\end{figure}

We continue in this fashion until all vertices of $T$ have been considered. We create such a list for each tree of $A'$, and then write them successively one after the other in one long list, and finally insert any remaining vertices of $A$ in any order. We have obtained the desired list with the elements of $V(A)$.

Let
\[[b_1,b_2,\dots,b_n]\]
be the base of the pyramid $B$. Consider any map $f$ from $V(A)$ to  $\{b_1,b_2,\dots,b_n\}$ such that the images of elements of $V(A)$ ordered according to our list have ascending indices. By Theorem \ref{thm:siepol}, $A\t B$ is a polyhedron.
\end{proof}

On the other hand, let $A$ be the triangular bipyramid, so that with the notation of Corollary \ref{cor:tree}, $A'$ is not a tree. If $f$ is a locally injective function, then there does not exist any graph $B$ such that $A\t B$ is a polyhedron. Indeed, let us label $u,v,w$ the vertices of degree $4$ in $A$, and $x,y$ those of degree $3$. According to Theorem \ref{thm:siepol}, there exists a region $\calR_x$ of $B$ containing $f(u),f(v),f(w)$. Moreover, $B$ has a region $\calR_u$ containing $f(x),f(v),f(y),f(w)$ in this order, a region $\calR_v$ containing $f(x),f(u),f(y),f(w)$ in this order, and a region $\calR_w$ containing $f(x),f(u),f(y),f(v)$ in this order. Since each pair of $\calR_u,\calR_v,\calR_w$ share three vertices, in fact they coincide, thus $f(u),f(v),f(w),f(x),f(y)$ all lie on one region $\calR$ of $B$. However, there is no ordering of these five vertices around $\calR$ that respects all the above conditions.

\section{Classification of regular, planar Sierpi\'{n}ski products}
\label{sec:A}

Sections \ref{sec:A} and \ref{sec:class} develop the theory necessary to prove Theorem \ref{thm:class} classifying the regular polyhedral Sierpi\'{n}ski products.

\subsection{Preliminaries on $\ca(n_1,n_2,\dots,n_k)$}
Let us consider $\ca(n_1,n_2,\dots,n_k)$ as in Definition \ref{def:A}. By planarity we have
\begin{equation}
	\label{eq:sum5}
n_1+n_2+\dots+n_k\leq 5.
\end{equation}
We will begin by showing that $k\leq 3$ always holds.
\begin{prop}
	\label{prop:emptya}
The classes $\ca(2,1,1,1)$, $\ca(1,1,1,1)$, and $\ca(1,1,1,1,1)$ are empty. In particular, if $A\in\ca(n_1,n_2,\dots,n_k)$, then $k\leq 3$.
\end{prop}
\begin{proof}
By contradiction, let $A\in\ca(2,1,1,1)$. Then $A$ is a $5$-regular, connected, plane $(p,q)$-graph on $r$ regions, with colours assigned to the vertices so that each has neighbours coloured $c_1,c_1,c_2,c_3,c_4$ in this cyclic order. Counting the neighbours of each colour for all vertices, we deduce that in $A$ there are $2p/5$ vertices of colour $c_1$, and $p/5$ of each other colour.

We wish to give an upper bound to the number $r_3$ of triangular regions in $A$. First, no triangular region may have two vertices of the same colour different from $c_1$, otherwise the third vertex would have two neighbours of the same colour different from $c_1$. Second, no triangular region may have two vertices of colours $c_i,c_j$ where $j-i\equiv 2\pmod 4$, otherwise these would appear consecutively in the cyclic order of neighbours around the third vertex. Thereby, a triangular region in $A$ may only have vertices of colours
\[c_1,c_1,c_1, \quad c_1,c_1,c_2, \quad\text{or}\quad c_1,c_1,c_4.\]
Recalling that the cyclic order of colours of neighbours around each vertex is $c_1,c_1,c_2,c_3,c_4$, it follows that every vertex of colour $c_1$ belongs to at most three triangular regions, every vertex of colour $c_2$ or $c_4$ belongs to at most one triangular region, and vertices of colour $c_3$ cannot belong to any triangular regions, as illustrated in Figure \ref{fig:illu}.
\begin{figure}[ht]
\centering
\includegraphics[width=3.5cm]{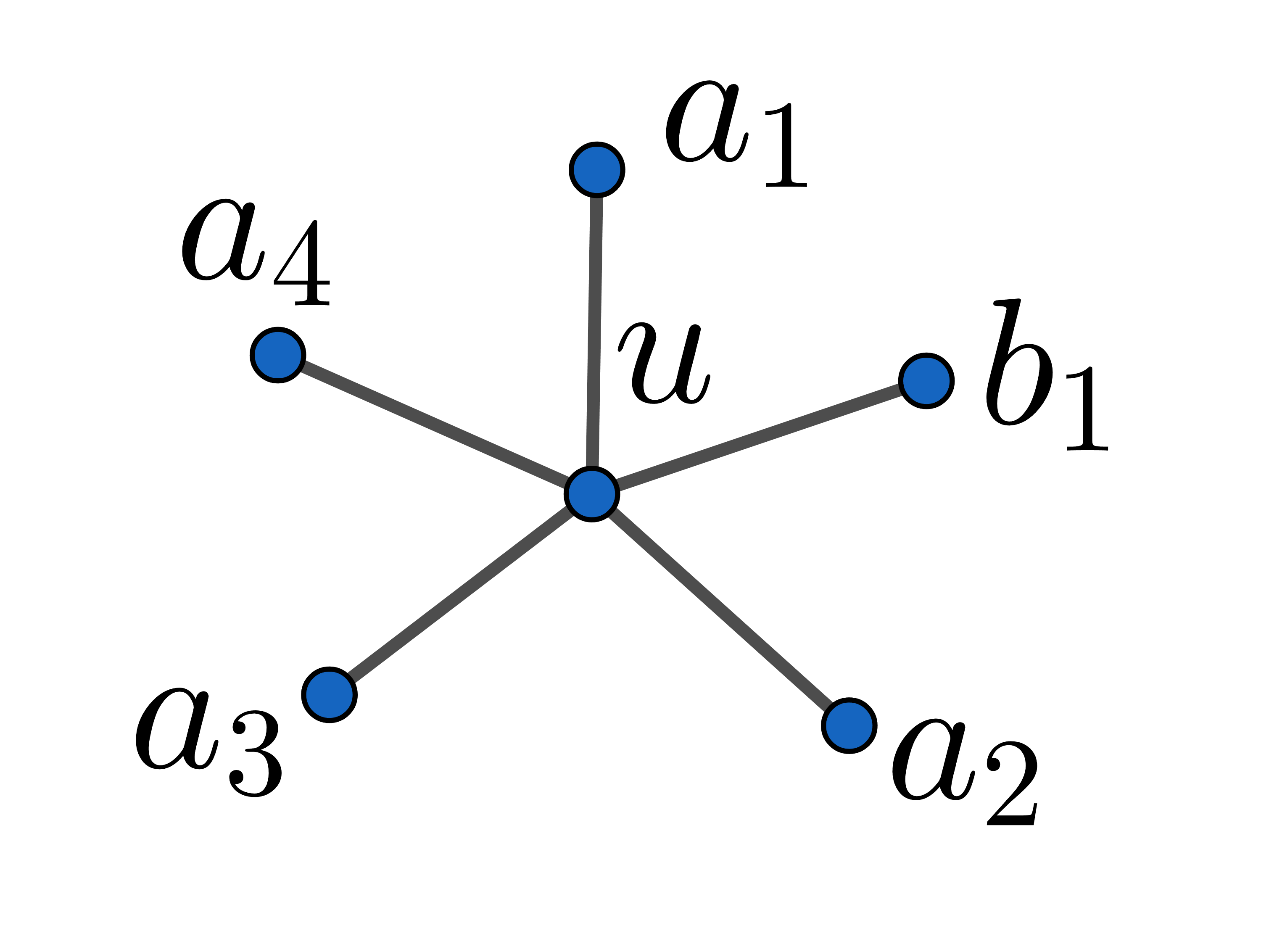}
\caption{The vertices $a_i$ have the colour $i$, and $b_1$ has the colour $1$. If $u$ has the colour $2$, then e.g.\ $b_1,a_2$ are not adjacent, otherwise $b_1$ would have two neighbours of colour $c_2$.}
\label{fig:illu}
\end{figure}

Therefore,
\[3r_3=\sum_{v\in V(G)}|\{\text{triangular regions containing }v\}|\leq 3\cdot\frac{2p}{5}+1\cdot\frac{p}{5}+0\cdot\frac{p}{5}+1\cdot\frac{p}{5}=\frac{8p}{5}.\]
By the handshaking lemma, $5p=2q$, so that by Euler's formula the total number of regions in $A$ is $3p/2+2$. We deduce a lower bound for the number of non-triangular regions,
\[r-r_3=\frac{3p}{2}+2-r_3\geq\frac{3p}{2}+2-\frac{8p}{15}=\frac{29p}{30}+2.\]
In any planar graph, the number of edges $q$ cannot exceed $3p-6$ minus the number of non-triangular regions. We obtain
\[\frac{5p}{2}=q\leq 3p-6-\left(\frac{29p}{30}+2\right),\]
leading to $p\leq -120/7$, contradiction. Hence $\ca(2,1,1,1)=\emptyset$.

Now let $A\in\ca(1,1,1,1)$ or $A\in\ca(1,1,1,1,1)$. The neighbours of each $a\in V(A)$ have the colours $c_1,c_2,c_3,c_4$ or $c_1,c_2,c_3,c_4,c_5$ in cyclic order. We construct a cycle $C$ of length $\ell\geq 4$ in $A$ in the following way. We take any vertex $u_1$ of colour $c_1$, its neighbour $u_2$ of colour $c_2$, and for each $u_i$, $i\geq 2$ of colour $c_i$, its neighbour $u_{i+1}$ of colour $c_{(i\mod 4)+1}$. Each vertex has only one neighbour of each colour, hence there is only one possible choice of vertex to add at every step when building $C$. Since $A$ is finite, we eventually return to $u_1$.

Now $u_2$ has neighbours $u_1,u_{3}$ of colours $c_1$ and $c_3$ on $C$, and $v,w$ respectively of colours $c_2$ and $c_4$, on opposite sides of $C$ (one internal and one external). We consider a path $P_2$ starting at $u_2$ where the sequence of colours of the vertices is
\[c_2,c_2,c_4,c_4,c_2,c_2,c_4,c_4,\dots,c_2,c_2,c_4,c_4.\]
Again since each vertex has only one neighbour of each colour, there is only one possible choice of vertex to add at every step when building $P_2$. Since $A$ is finite, we eventually return to $u_2$, and the last vertex added to $P_2$ before returning to $u_2$ is $w$. As $v,w$ are on opposite sided of $C$, by planarity $P_2$ must necessarily contain an element of $C$ other than $u_2$, say $u_{j_2}$, where $j_2\geq 4$. The reader may refer to Figure \ref{fig:1234}.
\begin{figure}[ht]
\centering
\includegraphics[width=4.5cm]{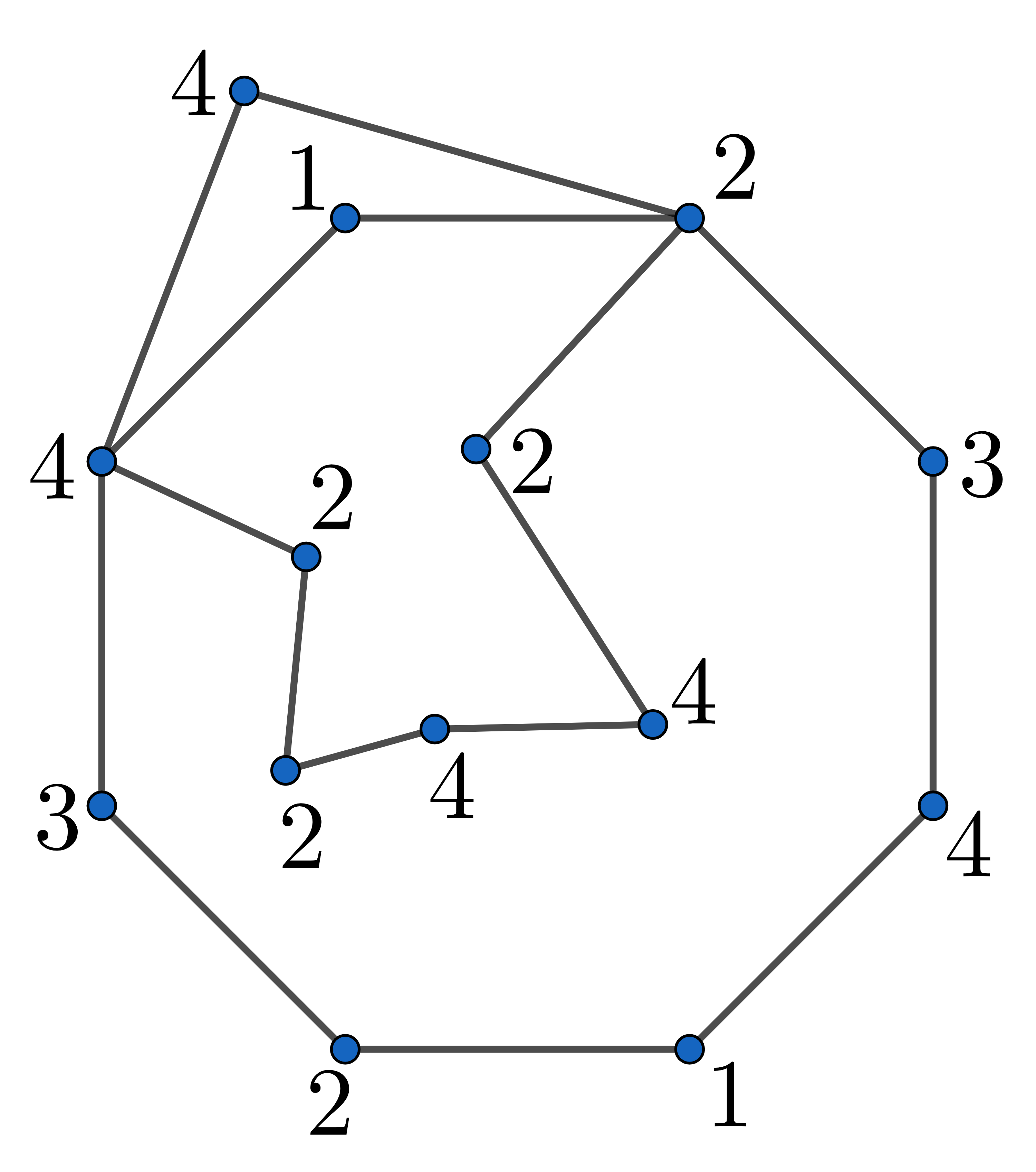}
\caption{Labels represent colours of vertices. Here $j_2=8$.}
\label{fig:1234}
\end{figure}

Next, we consider a path $P_3$ starting at $u_3$, 
where the sequence of colours of the vertices is
\[c_3,c_3,c_1,c_1,c_3,c_3,c_1,c_1,\dots,c_3,c_3,c_1,c_1.\]
The path $P_3$ must necessarily contain an element of $C$ other than $u_3$, say $u_{j_3}$, where $j_3\geq 5$. By construction, $P_2$ and $P_3$ are disjoint. Hence by planarity, they cannot cross. Therefore, $j_3<j_2$. Next, we consider a path $P_4$ starting at $u_4$, and we continue in this fashion. We obtain a decreasing sequence
\[\ell\geq j_2>j_3>j_4>\dots>j_i\geq i+2\]
for every $i\geq 2$, contradiction. Hence $\ca(1,1,1,1)$ and $\ca(1,1,1,1,1)$ are indeed empty.

We point out a couple of alternative arguments to prove that $\ca(1,1,1,1,1)$ is empty. Given a graph in $\ca(1,1,1,1,1)$, by recolouring the $c_5$ vertices  with the colour $c_1$, we get a member of $\ca(2,1,1,1)$, that we have already ruled out. Alternatively, replicating the above argument for $\ca(2,1,1,1)=\emptyset$, one shows that a member of $\ca(1,1,1,1,1)$ is a $5$-regular, planar graph with no triangular regions, which is impossible.
\end{proof}

Henceforth we will refer interchangeably to the three colours as $c_1,c_2,c_3$ or as red, blue, green respectively. The following lemma shows that in every element of $\ca(n_1,n_2,\dots,n_k)$, around each vertex the neighbours of the same colour are consecutive in the cyclic order.
\begin{lemma}
	\label{le:cy}
	Let $A\in\ca(n_1,n_2,\dots,n_k)$. Then for every $a\in V(A)$ the cyclic order of the colours of the neighbours arranged around $a$ is
	\[\underbrace{c_1,c_1,\dots,c_1}_{n_1},\underbrace{c_2,c_2,\dots,c_2}_{n_2},\dots,\underbrace{c_k,c_k,\dots,c_k}_{n_k}.\]
\end{lemma}
\begin{proof}
	By contradiction, around every vertex we can find neighbours of colours $c_1,c_2,c_1,c_i$, $i\neq 1$ in this cyclic order (if $A$ is $5$-regular there is also a fifth neighbour). The subgraph of $A$ formed by the red vertices is either $2$- or $3$- regular, thus by finiteness of $A$ there exists a cycle
	\[C: u_1,u_2,\dots,u_n, \quad n\geq 3\]
	of red vertices such that inside $C$ there are no red vertices. Now $u_1$ has a neighbour $v$ inside $C$ (that is not red). Moreover, $v$ is adjacent to two red vertices, and since there are no red vertices inside of $C$, then by planarity in fact $v$ is adjacent to two vertices of $C$, say $u_1,u_j$, $2\leq j\leq n$. Now $v$ has non-red neighbours $y,z$, one each inside of the cycles
	\[u_1,v,u_j,u_{j-1},\dots,u_2,\]
	and
	\[u_1,v,u_j,u_{j+1},\dots,u_n.\]
	The edge $vy$ lies on a cycle where the colours of vertices are 
	\[c_2,c_2,c_i,c_i,c_2,c_2,c_i,c_i,\dots,c_2,c_2,c_i,c_i\]
	in order, and the same may be said for the edge $vz$. Since $4\leq \deg_A(v)\leq 5$ and $A$ is planar, this situation is impossible. The reader may refer to Figure \ref{fig:cy}.
	\begin{figure}[ht]
		\centering
		\includegraphics[width=4.5cm]{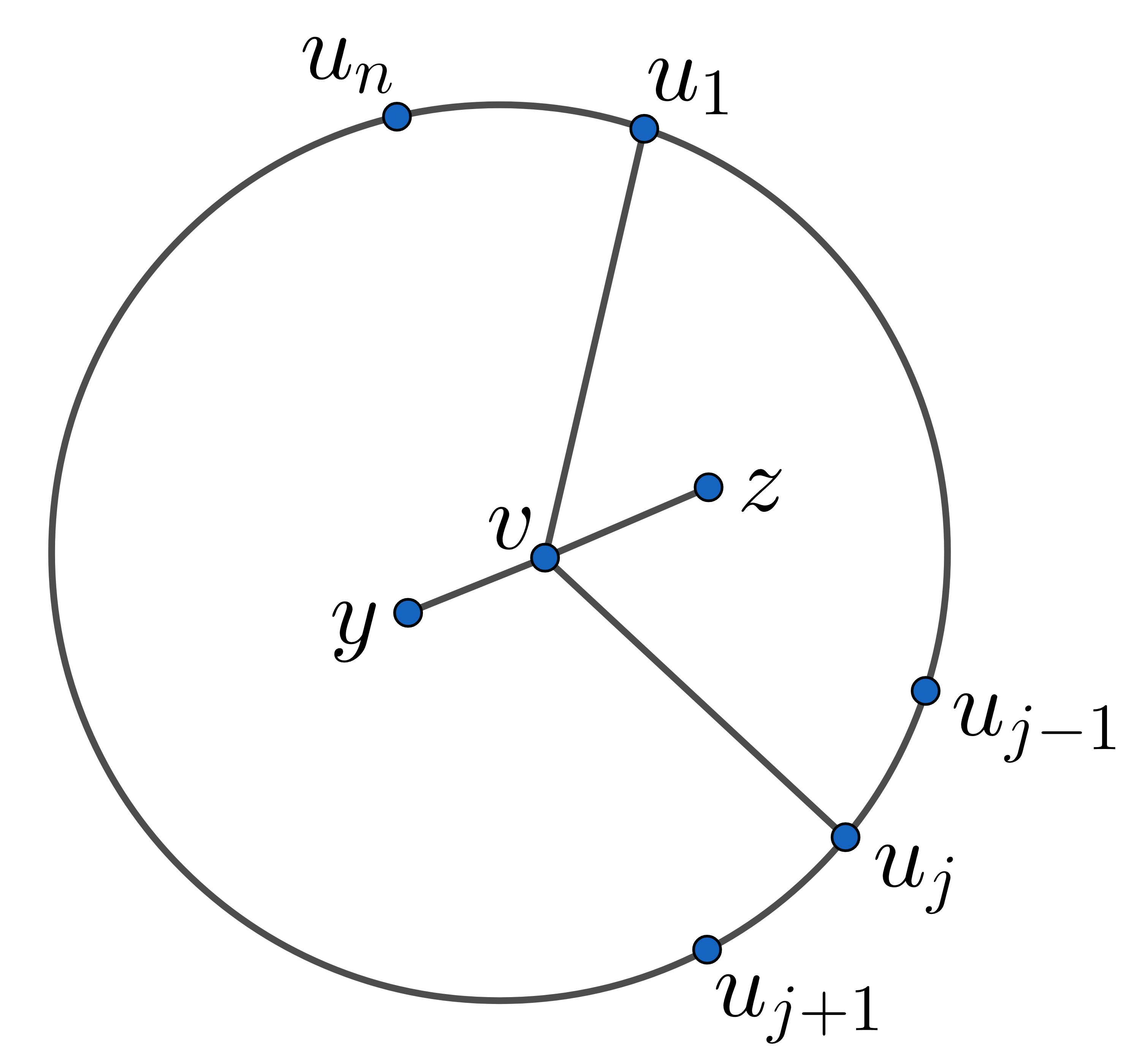}
		\caption{Lemma \ref{le:cy}.}
		\label{fig:cy}
	\end{figure}
	
\end{proof}

Let $A\in\ca(n_1,n_2,\dots,n_k)$. By Proposition \ref{prop:emptya}, we know that $k\leq 3$. Bearing in mind \eqref{eq:sum5}, for $k=1$ we have
\[(n_1)\in\{(1), (2), (3), (4), (5)\},\]
for $k=2$ we have
\[(n_1,n_2)\in\{(1,1), (2,1), (2,2), (3,1), (3,2), (4,1)\},\]
and for $k=3$ we have
\[(n_1,n_2,n_3)\in\{(1,1,1), (2,1,1), (2,2,1), (3,1,1)\}.\]
By Lemma \ref{le:cy}, in $\ca(2,1,1)$ around each vertex the cyclic order of colours of neighbours is red, red, blue, green, and in $\ca(2,2)$ it is red, red, blue, blue.

Let $A\in\ca(n_1,n_2,n_3)$. We change all vertices of colour $c_3$ to colour $c_1$, to see that $A\in\ca(n_1+n_3,n_2)$. That is to say,
$\ca(n_1,n_2,n_3)\subseteq\ca(n_1+n_3,n_2)$. By similar considerations, $\ca(n_1+n_2)\subseteq\ca(n_1+n_2)$, and for $n_1\geq n_2+n_3$, we have $\ca(n_1,n_2,n_3)\subseteq\ca(n_1,n_2+n_3)$. For instance, as mentioned in the introduction if in the member of $\ca(2,1,1)$ in Figure \ref{fig:A211} we replace all labels for the colour $c_3$ with $c_2$ we obtain a member of $\ca(2,2)$. Hence we have the chains of inclusions
\begin{align*}
	&\ca(1,1)\subseteq\ca(2),
	\\&\ca(1,1,1)\subseteq\ca(2,1)\subseteq\ca(3),
	\\&\ca(2,1,1)\subseteq\ca(2,2),\ca(3,1)\subseteq\ca(4),
	\\&\ca(2,2,1),\ca(3,1,1)\subseteq\ca(3,2),\ca(4,1)\subseteq\ca(5).
\end{align*}

Each of these classes of graphs appears in the construction of regular, connected, planar Sierpi\'{n}ski products (Appendix \ref{sec:appb}). Note that $\ca(n_1)$ is simply the full class of connected, planar, $n_1$-regular graphs. If $A\in\ca(1,1)$, then $A$ is simply a cycle of length a multiple of $4$, where the order of colours around the cycle is
\[c_1,c_1,c_2,c_2,c_1,c_1,c_2,c_2,\dots,c_1,c_1,c_2,c_2.\]

In Section \ref{sec:AA}, we will prove that $\ca(3,2)$ is empty, hence its subsets $\ca(2,2,1)$ and $\ca(3,1,1)$ are also empty. We will inspect properties for the remaining classes,
\begin{align}
	\label{eq:inf0}
	\notag&\ca(1,1,1)\subseteq\ca(2,1),
	\\\notag&\ca(2,1,1)\subseteq\ca(2,2),\ca(3,1),
	\\&\ca(4,1).
\end{align}
Apart from being intriguing on their own account, they will be used in the proof of Theorem \ref{thm:class}.

\subsection{Characterisation of regular, connected, planar Sierpi\'{n}ski products}
The following result is a necessary and sufficient condition on $A,B,f$ for $A\t B$ to be a regular, connected, planar Sierpi\'{n}ski product. The classification of these graphs is tabulated in Appendix \ref{sec:appb}. Recall Definitions \ref{def:A} for $\ca(n_1,n_2,\dots,n_k)$, \ref{def:B} for $\cb_r(n_1,n_2,\dots,n_k)$, and \ref{def:col} for a function that preserves colours.

\begin{prop}
	\label{prop:col}
	Let $A,B$ be graphs and $f$ a function. Then $A\t B$ is $r$-regular, connected, and planar if and only if there exist $1\leq k\leq 3$ and $n_1\geq n_2\geq\dots\geq n_k\geq 1$ such that
	\[A\in\ca(n_1,n_2,\dots,n_k), \qquad B\in\cb_r(n_1,n_2,\dots,n_k),\]
	and $f$ preserves colours.
\end{prop}
\begin{proof}
$\Rightarrow$. Let $A\t B$ be $r$-regular, connected, and planar. By definition of Sierpi\'{n}ski product, for every $a\in V(A)$ and $b\in V(B)$ we have
	\begin{align}
		\label{eq:sumt}
		r&\notag=\deg_{A\t B}(a,b)=
		\\&=\deg_B(b)+|\{a': aa'\in E(A) \text{ and } f(a')=b\}|.
	\end{align}
	In particular, each vertex of $B$ is of degree at most $r$, and it belongs to the image of $f$ if and only if it is of degree at most $r-1$ in $B$. We assign a different arbitrary colour to each $b\in V(B)$ of degree at most $r-1$. We then assign the same colour as $b$ to every $a\in V(A)$ such that $f(a)=b$. By \eqref{eq:sumt}, given $a\in V(A)$, for every $b\in V(B)$ satisfying
	\[\deg_B(b)=j\leq r-1,\]
	the vertex $a$ has $r-j$ neighbours of the colour of $b$. In particular, $A$ is a regular graph, and each $a\in V(A)$ has neighbours of the same combination of colours. We label the colours $c_1,c_2,\dots,c_k$ in cyclic order, so that for each $a\in V(A)$, there are $n_i$ neighbours coloured $c_i$, and $n_1\geq n_2\geq\dots\geq n_k$. That is to say, $A\in\ca(n_1,n_2,\dots,n_k)$. By Proposition \ref{prop:emptya}, we have $k=|\ii(f)|\leq 3$.
	
	As for $B$, we claim that the coloured vertices (i.e., those in the image of $f$) all lie on the same region. By contradiction, we take $a\in V(A)$ and sketch $aB$ in the plane so that the vertex $x_1$ of colour $c_1$ is on the boundary of the external region $\calR$, and say $x_2$ of colour $c_2$ is not. Let $a_1,a_2$ be neighbours of $a$ in $A$ of colours $c_1,c_2$ respectively. Then in $A\t B$ the subgraph $a_1B$ lies outside of the boundary of $\calR$ while $a_2B$ lies inside. Hence the boundary of $\calR$ in $aB$ is a separating cycle of $A\t B$. By definition of Sierpi\'{n}ski product, $a$ is a separating vertex of $A$. It follows that all vertices of $A$ are separating vertices, contradiction. We deduce that $B\in\cb_r(n_1,n_2,\dots,n_k)$.
	
	$\Leftarrow$. Vice versa, assume that $A\in\ca(n_1,n_2,\dots,n_k)$, $B\in\cb_r(n_1,n_2,\dots,n_k)$, and $f$ preserves colours, and let us show that $A\t B$ is a regular, connected, planar graph. Regularity follows readily from the assumptions on $A,B,f$. Since $A,B$ are connected, connectivity of $A\t B$ follows from Lemma \ref{le:deg3}. Next, $B$ has a planar embedding such that the coloured vertices (i.e., the elements of $\ii(f)$) all lie on the same region. We sketch the graphs
	\[\{aB : a\in V(A)\},\]
	in the plane so that in each, the external region is the one containing the coloured vertices.
	We then add the remaining edges of $A\t B$. We contact each $aB$ to a single vertex, obtaining a graph isomorphic to $A$. Since $A$ is planar, then so is $A\t B$.
\end{proof}


\subsection{Classification and structural properties of $\ca(n_1,n_2,\dots,n_k)$}
\label{sec:AA}

In this section, we inspect the fine properties of the classes \eqref{eq:inf0}
\begin{align}
	\label{eq:inf}
	\notag&\ca(1,1,1)\subseteq\ca(2,1),
	\\\notag&\ca(2,1,1)\subseteq\ca(2,2),\ca(3,1),
	\\&\ca(4,1).
\end{align}

\begin{lemma}
Each of the classes in \eqref{eq:inf} contains infinitely many graphs.
\end{lemma}
\begin{proof}
By the inclusions \eqref{eq:inf}, it suffices to show that $\ca(1,1,1)$, $\ca(2,1,1)$, and $\ca(4,1)$ are infinite. Elements of these classes may be found in Figures \ref{fig:A111}, \ref{fig:A211}, and \ref{fig:A41} respectively. To build new graphs, one may proceed as follows. One sketches two copies of a graph $A$ in the desired class $\ca$, one with an edge $a_1a_2$ on the external region, and the other with an edge $a_3a_4$ on the external region, such that $a_1,a_2,a_3,a_4$ are of colour $c_1$. Then the graph
\[A-a_1a_2-a_3a_4+a_1a_3+a_2a_4\]
is a new element of $\ca$. Note that it is connected and planar by construction.
\begin{figure}[ht]
\centering
\includegraphics[width=6.25cm]{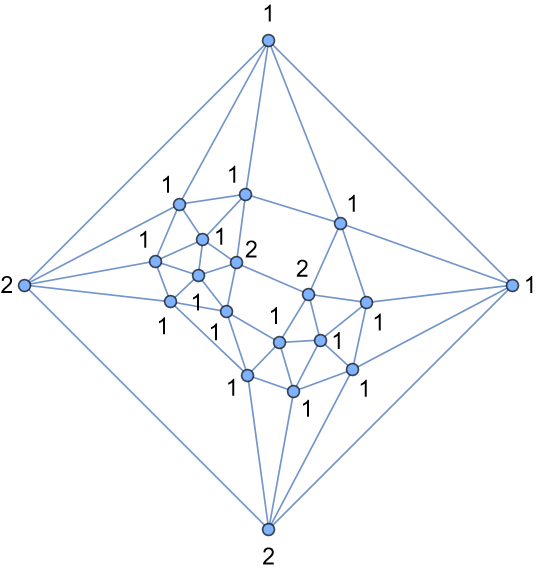}
\caption{An element of $\ca(4,1)$.}
\label{fig:A41}
\end{figure}
\end{proof}

The following remark records some useful properties of subgraphs of $A\in\ca(n_1,n_2,\dots,n_k)$ where we consider a subset of the colours of the vertices.
\begin{rem}
	\label{rem:A}
Let $A\in\ca(n_1,1)$. Since each vertex has $n_1$ neighbours of colour $c_1$, then the subgraph of $A$ generated by the vertices of colour $c_1$ is $n_1$-regular, connected, and planar. For $n_1=2$, it is a disjoint union of cycles.
\\
Similarly, for $A\in\ca(1,1,1)$, fixing two colours $c_i\neq c_j$, $i,j\in\{1,2,3\}$, the subgraph of $A$ generated by the vertices of colours $c_i,c_j$ is a disjoint union of cycles each of length a multiple of $4$, where the order of colours around each cycle is
\[c_i,c_i,c_j,c_j,c_i,c_i,c_j,c_j,\dots,c_i,c_i,c_j,c_j.\]
\\
Let $A\in\ca(2,1,1)$. The subgraph of $A$ generated by the vertices of colour $c_1$ is a disjoint union of cycles, and the subgraph generated by the vertices of colours $c_2,c_3$ is a disjoint union of cycles each of length a multiple of $4$, where the order of colours around each cycle is
\[c_2,c_2,c_3,c_3,c_2,c_2,c_3,c_3,\dots,c_2,c_2,c_3,c_3.\]
\\
Let $A\in\ca(2,2)$. For $i=1,2$, the subgraph of $A$ generated by the vertices of colour $c_i$ is a disjoint union of cycles.
\end{rem}

For instance, in Figure \ref{fig:A41} the subgraph generated by the vertices of colour $c_1$ is the octagonal antiprism, a $4$-regular, connected, planar graph.

\begin{lemma}
	\label{le:A111}
	Let $A\in\ca(1,1,1)$. Then $A$ is $2$-connected.
\end{lemma}
\begin{proof}
By contradiction, there exists a separating vertex $v$ in $A$. Since $\deg(v)=3$, then at least one edge incident to $v$, say $vw$, is a bridge in $A$. Let $c_i$ be the colour of $v$ and $c_j$ the colour of $w$. As seen in Remark \ref{rem:A}, the edge $vw$ lies on a cycle where the order of colours is
\[c_i,c_i,c_j,c_j,c_i,c_i,c_j,c_j,\dots,c_i,c_i,c_j,c_j.\]
Then $vw$ is a bridge that lies on a cycle, contradiction.
\end{proof}
Figure \ref{fig:A111} illustrates an element of $\ca(1,1,1)$ of connectivity $3$, and Figure \ref{fig:A111conn2} illustrates an element of $\ca(1,1,1)$ of connectivity $2$. On the other hand, there exist elements of $\ca(2,1)$ of connectivity $1$, as in Figure \ref{fig:A21conn1}.
	\begin{figure}[ht]
		\centering
		\begin{subfigure}{0.47\textwidth}
			\centering
			\includegraphics[width=3.75cm]{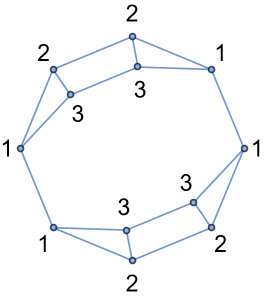}
			\caption{A member of $\ca(1,1,1)$ of connectivity $2$.}
			\label{fig:A111conn2}
		\end{subfigure}
		\hfill
		\begin{subfigure}{0.47\textwidth}
			\centering
			\includegraphics[width=4.cm]{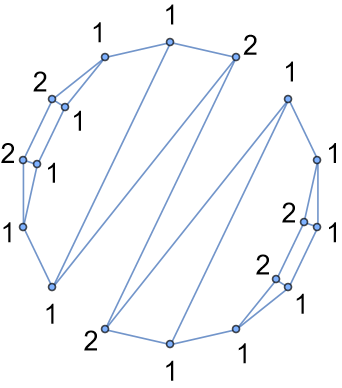}
			\caption{A member of $\ca(2,1)$ of connectivity $1$.}
			\label{fig:A21conn1}
		\end{subfigure}
		\caption{Labels represent colours of vertices.}
		\label{fig:A02}
	\end{figure}

\begin{lemma}
	\label{le:111}
There exist infinitely many $3$-connected graphs in $\ca(1,1,1)$.
\end{lemma}
\begin{proof}
One starts with the triangular prism, depicted in Figure \ref{fig:A111}. One obtains new graphs of the desired type by starting from a $3$-connected member of $\ca(1,1,1)$, and then replacing a subgraph as in Figure \ref{fig:A111ta} with the graph in Figure \ref{fig:A111tb}. Note that this transformation preserves planarity and $3$-connectivity.
	\begin{figure}[ht]
		\centering
		\begin{subfigure}{0.47\textwidth}
			\centering
			\includegraphics[width=2.75cm]{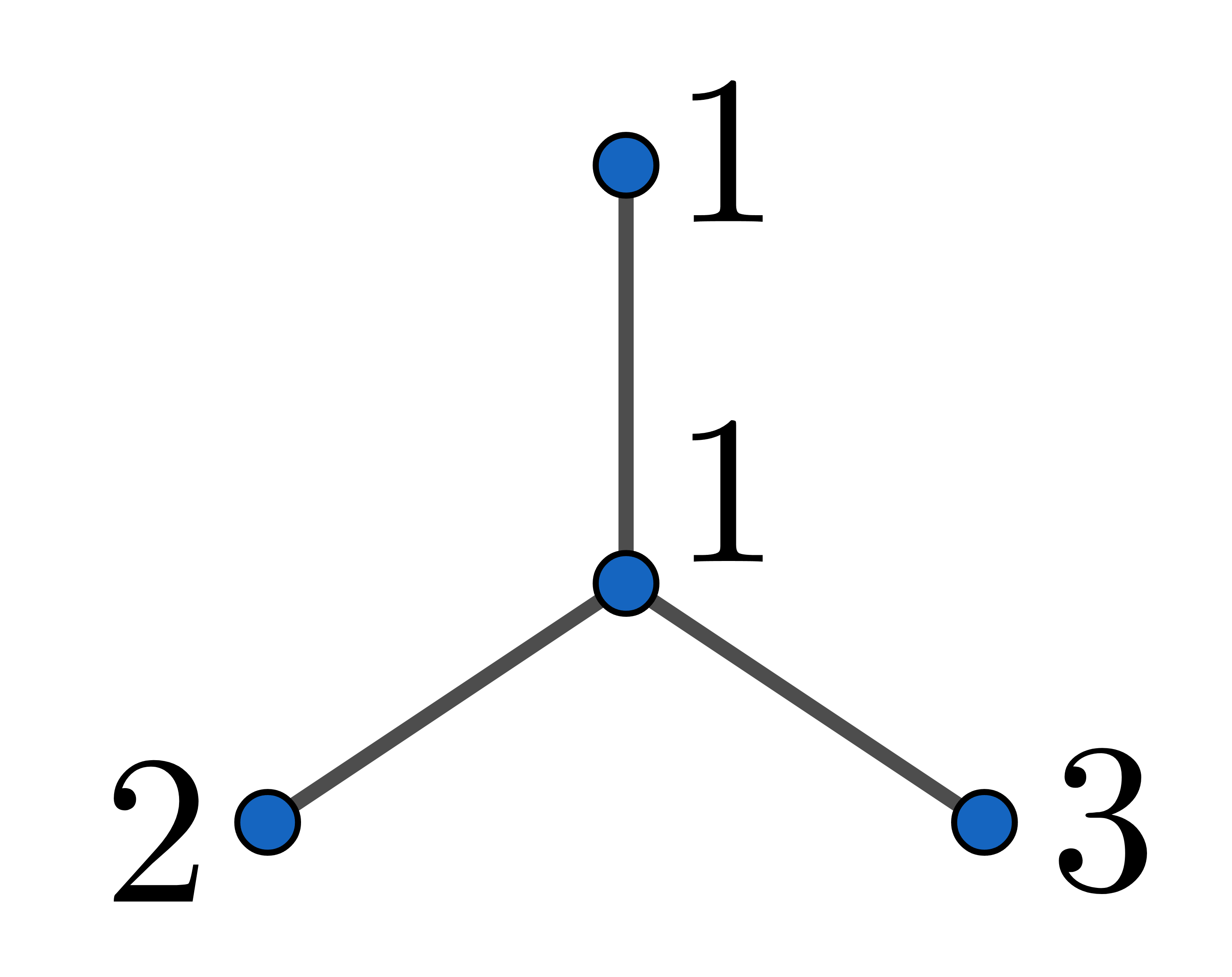}
			\caption{}
			\label{fig:A111ta}
		\end{subfigure}
		\hspace{-0.75cm}
		\begin{subfigure}{0.47\textwidth}
			\centering
			\includegraphics[width=4.cm]{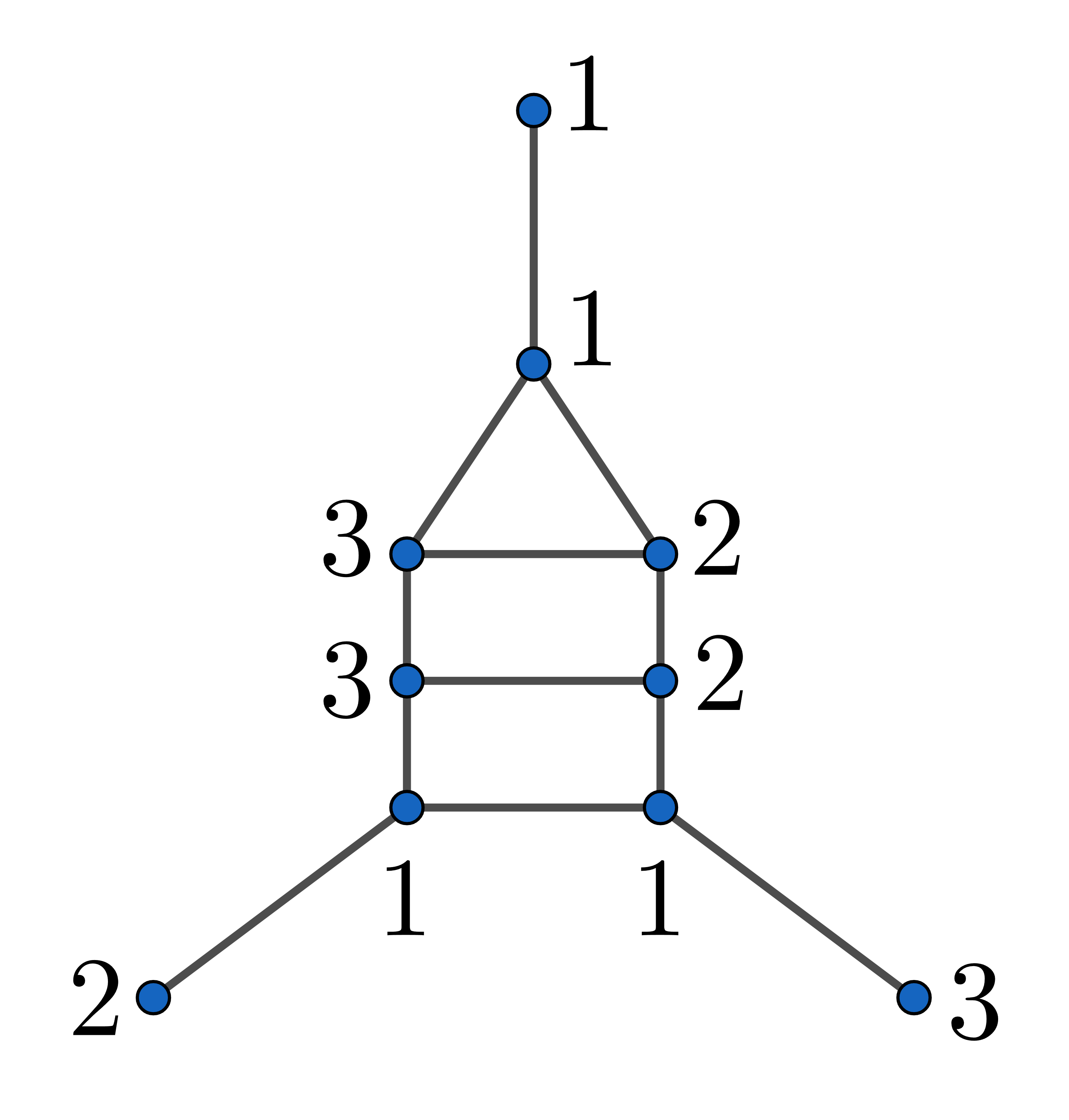}
			\caption{}
			\label{fig:A111tb}
		\end{subfigure}
		\caption{Transformation to obtain $3$-connected elements of $\ca(1,1,1)$.}
		\label{fig:A111t}
	\end{figure}
	
\end{proof}

There exist elements of $\ca(2,1,1)$ of connectivity $1$ (e.g.\ we have found an example on $80$ vertices, data available on request). Since $\ca(2,2)$ and $\ca(3,1)$ both contain $\ca(2,1,1)$, these classes of graphs also contain members of connectivity $1$. We record that since these are $4$-regular graphs, they are bridgeless, thus every edge lies on a cycle.

Thanks to Lemma \ref{le:111} and to the following result, one may construct infinitely many solutions $A$ for each of the six scenarios in Theorem \ref{thm:class}. Infinitely many solutions $B$ for each scenario are constructed in Appendix \ref{sec:appb}.
\begin{lemma}
	There exist infinitely many $3$-connected graphs in $\ca(2,2)$, and infinitely many $2$-connected graphs in $\ca_\#(2,1,1)$.
\end{lemma}
\begin{proof}
It suffices to construct $3$-connected members of $\ca(2,1,1)$, since then exchanging all green vertices with blue ones yields $3$-connected graphs in $\ca(2,2)$. We start with two copies of Figure \ref{fig:A211}, where $u_i,v_i$ are the vertices of colour $c_i$ on the external region of one copy, and $w_i,x_i$ the vertices of colour $c_i$ on the external region of the other copy, $i=2,3$. We then perform the operation
\[-u_2v_2-w_2x_2+u_2w_2+v_2x_2-u_3v_3-w_3x_3+u_3w_3+v_3x_3\]
to obtain another $3$-connected member of $\ca(2,1,1)$.
\end{proof}

We now turn to the graphs $\ca(2,2)$, that appear in scenario 1 of Theorem \ref{thm:class}. Given a plane graph $G$, the vertex $v$ expansion operation, described in \cite{tutt61}, is defined by
\[G-v+xy+xu_1+\dots+xu_{m}+yu_{m+1}+\dots+yu_{n},\]
where
\[u_1,\dots,u_n\]
are the neighbours of $u$ in cyclic order. The fact that this operation preserves $3$-connectivity is key in \cite{tutt61}.
\begin{lemma}
	\label{le:A22conn}
	Let $A\in\ca(2,2)$, $B=K_2\in\cb_3(2,2)$, and $f:V(A)\to V(B)$ be the function that preserves colours. Then $A\t B$ is obtained from $A$ by expanding all the vertices. In particular, for $k=2,3$, $A\t B$ is $k$-connected if and only if $A$ is $k$-connected.
\end{lemma}
\begin{proof}
	Let us construct $A\t B$ using the definition of Sierpi\'{n}ski product, starting from $A$. We replace each vertex $a$ with a copy of $K_2$ labelled $aK_2$, with one red and one blue vertex. Call $a',a''$ the red neighbours of $a$ in $A$, and $a''',a''''$ the blue neighbours. We connect the red vertex of $aK_2$ to the vertex of the same colour as $a$ in $a'K_2$ and $a''K_2$, and we connect the blue vertex of $aK_2$ to the vertex of the same colour as $a$ in $a'''K_2$ and $a''''K_2$. We note that this construction is equivalent to expanding all vertices of $A$.
	
	Let $k\in\{2,3\}$, and consider $a_1,a_2,\dots,a_k\in V(A)$. Since $A\t B$ is obtained from $A$ by expanding all vertices, then $\{a_1,a_2,\dots,a_k\}$ is a $k$-cut of $A$ if and only if
	\[\{(a_1,z_1),(a_2,z_2),\dots,(a_k,z_k)\}\]
	is a $k$-cut of $A\t B$ for every choice of $z_i\in V(a_iK_2)$, $1\leq i\leq k$ if and only if
	\[\{(a_1,z_1),(a_2,z_2),\dots,(a_k,z_k)\}\]
	is a $k$-cut of $A\t B$ for at least one choice of $z_i\in V(a_iK_2)$, $1\leq i\leq k$. Thus $A$ is $k$-connected if and only if $A\t B$ is $k$-connected, as claimed.
\end{proof}
The construction of Lemma \ref{le:A22conn} is illustrated in Figure \ref{fig:exp}.
	\begin{figure}[ht]
		\centering
		\begin{subfigure}{0.4\textwidth}
			\centering
			\includegraphics[width=3.5cm]{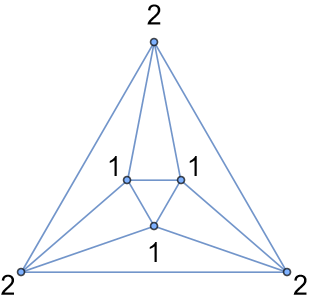}
			\caption{The octahedron $O\in\ca(2,2)$.}
			\label{fig:A22expand}
		\end{subfigure}
		\hfill
		\begin{subfigure}{0.58\textwidth}
			\centering
			\includegraphics[width=4.25cm]{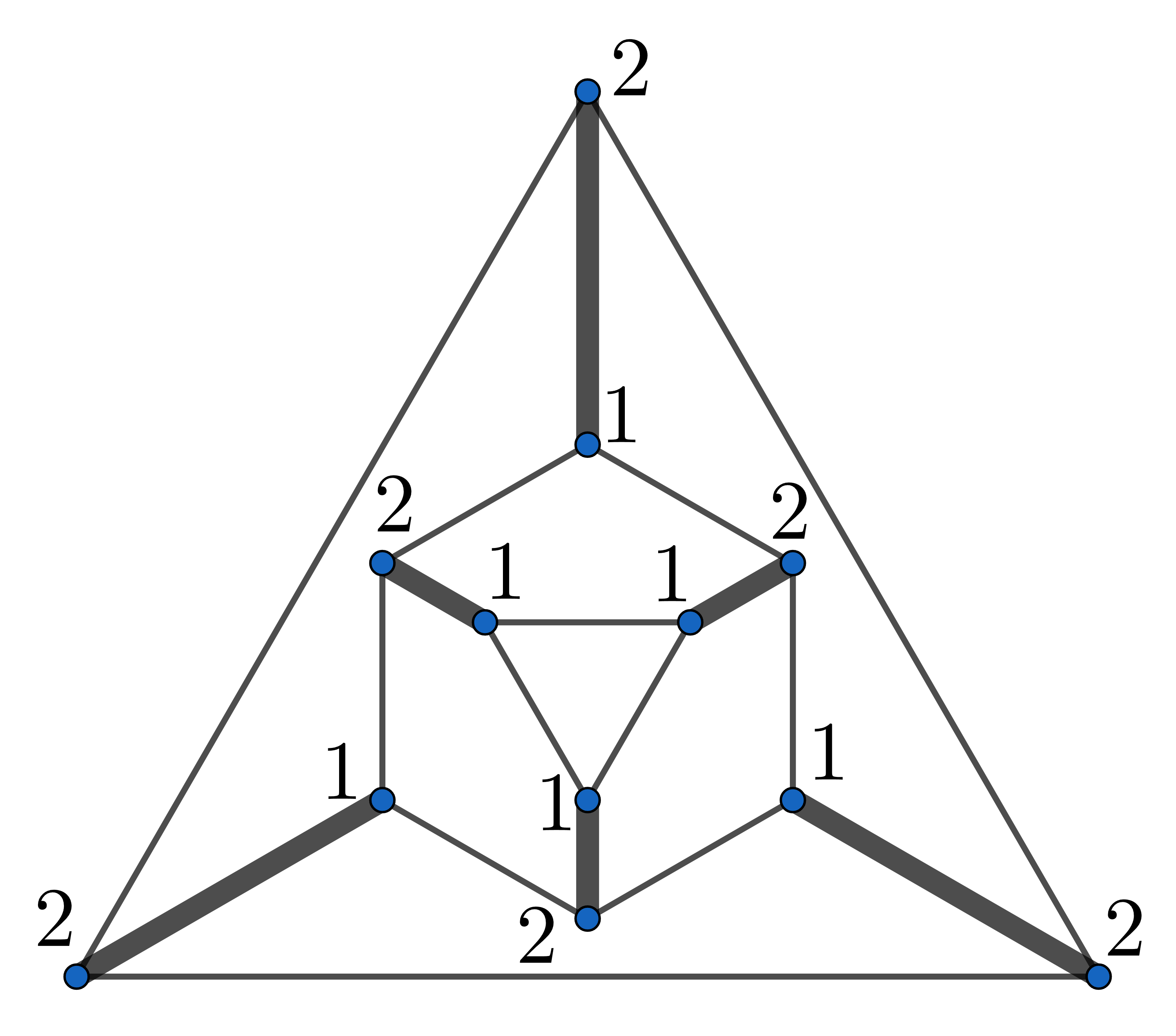}
			\caption{Expansion of $O$. Thick edges are the graphs $aK_2$, $a\in V(O)$.}
			\label{fig:A22expanded}
		\end{subfigure}
		\caption{Illustration of Lemma \ref{le:A22conn}. Labels represent colours of vertices.}
		\label{fig:exp}
	\end{figure}

\begin{lemma}
	\label{le:A22}
We have $A\in\ca(2,2)$ if and only if $A$ admits the following construction. One starts with a union of disjoint cycles, each of length at least $6$, with black edges and alternating red and blue vertices. Then one adds red cycles including only the red vertices, and blue cycles including only the blue ones, in such a way that the resulting graph is connected and planar, and moreover each red (resp. blue) vertex $v$ is incident to exactly two red (resp. blue) edges, and these two edges are on the same side (either both internal or both external) of the black cycle containing $v$.
\end{lemma}
\begin{proof}
To sketch $A$, we start with a path on black edges and alternating red and blue vertices. This path must be closed, since the graph is finite, and moreover it cannot repeat any vertices, since each vertex has exactly two neighbours of the opposite colour. Hence the path is in fact a cycle. We then continue by sketching a new black cycle along any leftover edges between vertices of opposite colour. The result is a disjoint union of black cycles and alternating red and blue vertices.

Next, we add the edges between vertices of the same colour. Each vertex has exactly two neighbours of the same colour, hence we will have red cycles of red vertices and blue cycles of blue vertices. By Lemma \ref{le:cy}, for each vertex $v$, its neighbours are arranged around it in the cyclic order red, red, blue, blue, hence the two incident edges of the same colour as $v$ are either both internal or both external to the black cycle containing $v$.

In the other direction, any such construction yielding a connected, planar graph generates an element of $A(2,2)$.
\end{proof}
For instance, with one black cycle of length $2n\geq 6$, one red internal cycle of length $n$, and one blue external cycle of length $n$, we obtain the $n$-gonal antiprism.

\begin{prop}
	\label{prop:A32}
	The classes $\ca(2,2,1)$, $\ca(3,1,1)$, and $\ca(3,2)$ are empty.
\end{prop}
\begin{proof}
	Since $\ca(2,2,1)\subseteq\ca(3,2)$ and $\ca(3,1,1)\subseteq\ca(3,2)$, it suffices to prove that $\ca(3,2)=\emptyset$. By contradiction, let $A\in\ca(3,2)$. Then $A$ is a $5$-regular, connected, plane graph, with colours assigned to the vertices so that each has neighbours coloured $c_1,c_1,c_1,c_2,c_2$ in this cyclic order, due to Lemma \ref{le:cy}. Let $A'$ be the subgraph of $A$ generated by the vertices of colour $c_1$. 
	Clearly $A'$ is a $3$-regular, plane graph (cf.\ Remark \ref{rem:A}).
	
	Consider a connected component $H$ of $A'$. We claim that $H$ is not outerplanar. Suppose by contradiction that it is. If $H$ is $2$-connected, then by Lemma \ref{le:out}, it contains at least two vertices of degree $2$, impossible. Hence $H$ cannot be $2$-connected. Let $H'$ be an endblock of $H$, and $v$ the only separating vertex of $H$ that lies on $H'$. Now on one hand $H'$ is $2$-connected and outerplanar, thus by Lemma \ref{le:out} it contains at least two vertices of degree $2$. On the other hand, since $A'$ is $3$-regular, then by construction in $H'$ every vertex except one is of degree $3$, contradiction. Hence every connected component of $A'$ is not outerplanar.
	
	We claim that there exists $u\in V(A')$ lying on a cycle of $A'$
	\[C: u_1,u_2,\dots,u_\ell, \quad\ell\geq 3\]
	such that inside $C$ one finds the two neighbours of $u$ of colour $c_2$, but no vertex of colour $c_1$. Indeed, taking a connected component $H$ of $A'$, since $H$ is not outerplanar, we may find $v_1\in V(H)$ not lying on the external region of $H$. Let $\calR_{v_1}$ be the region of $H$ such that the two neighbours of $v_1$ of colour $c_2$ lie inside of it (since the cyclic order of colours of neighbours around each vertex of $A$ is $c_1,c_1,c_1,c_2,c_2$, these two neighbours lie inside the same region of $H$). Then either we may take $u=v_1$, or there exists $v_2\in V(H)$ lying inside $\calR_{v_1}$. Note that $v_2$ does not lie on the external region of $H$. Let $\calR_{v_2}$ be the region of $H$ such that the two neighbours of $v_2$ of colour $c_2$ lie inside of it. Either $u=v_2$, or there exists $v_3\in V(H)$ lying inside $\calR_{v_2}$. Continuing in this fashion, by finiteness of $A$ we see that $u$ exists.
	
	Let $w_1$ be a neighbour of $u$ of colour $c_2$. Now $w_1$ has three neighbours of colour $c_1$. By planarity of $A$, and since there are no vertices of colour $c_1$ inside of $C$, these three neighbours all lie on $C$. Denote them by \[u_i,u_j,u_k, \quad 1\leq i<j<k\leq\ell,\]
	where $u\in\{u_i,u_j,u_k\}$. Since the cyclic order of colours of neighbours around each vertex of $A$ is $c_1,c_1,c_1,c_2,c_2$, and since $u_{i-1}$ and $u_{i+1}$ are of colour $c_1$, then the other neighbour $w_2$ of $u_j$ of colour $c_2$ lies inside of $C$, as in Figure \ref{fig:A32}. More precisely, $w_2$ lies either inside the cycle
	\[u_i,u_{i+1},\dots,u_j,w_1,\]
	or inside the cycle
	\[u_j,u_{j+1},\dots,u_k,w_1.\]
	Now $w_2$ has three neighbours of colour $c_1$, and by planarity of $A$ these are of the form
	$u_{i'},u_{j'},u_{k'}$, where either $i\leq i'<j'<k'=j$, or $j=i'<j'<k'\leq k$. In any case, we have
	\[k'-i'<k-i.\]
	\begin{figure}[ht]
	\centering
	\includegraphics[width=4.5cm]{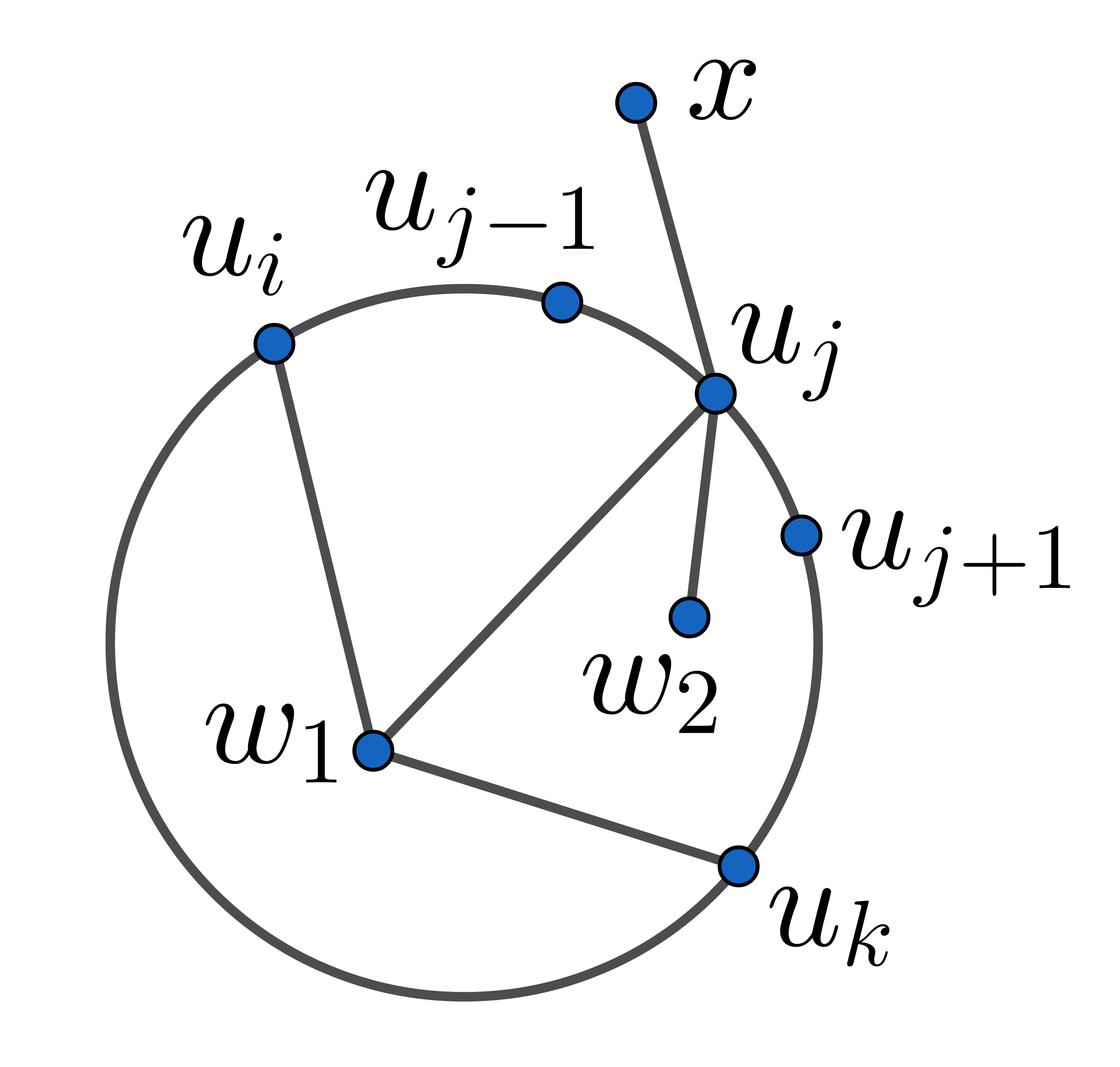}
	\caption{Proposition \ref{prop:A32}.}
	\label{fig:A32}
	\end{figure}
	
	Summarising, given $w_1$ inside $C$ of colour $c_2$ adjacent to $u_i,u_j,u_k$ with $1\leq i<j<k\leq\ell$, we have found a vertex $w_2$ inside $C$ of colour $c_2$ that is adjacent to $u_{i'},u_{j'},u_{k'}$, where $1\leq i'<j'<k'\leq\ell$ and $k'-i'<k-i$. Arguing in the same way, the other vertex $w_3$ of colour $c_2$ adjacent to $u_j'$ is adjacent to $u_{i''},u_{j''},u_{k''}$, with $1\leq i''<j''<k''\leq\ell$ and $k''-i''<k'-i'$, so that continuing in the same fashion we arrive at a contradiction by infinite descent. Thereby, indeed $\ca(2,2,1)=\ca(3,1,1)=\ca(3,2)=\emptyset$.
\end{proof}

\section{Classification of regular polyhedral Sierpi\'{n}ski products}
\label{sec:class}

\subsection{Auxiliary results}
The proof of Theorem \ref{thm:class} still requires a few preparatory considerations. Recall Definition \ref{def:b*} for $\cb_r^*(n_1,n_2,\dots,n_k)$. 

\begin{lemma}
	\label{le:B111}
Each member of $\cb_3^*(1,1,1)$ and $\cb_5^*(1,1,1)$ is $2$-connected.
\end{lemma}
\begin{proof}
Let $B\in\cb_r^*(1,1,1)$, with $r\in\{3,5\}$. By contradiction, if we have a separating vertex $b$ in $B$, then for each component $J$ of $B-b$ the RHS of \eqref{eq:hsum} equals $2$. For distinct components $J_1,J_2$ of $B-b$ we thus have
\[2+2\leq\sum_{v\in V(J_1)}(r-\deg_B(v))+\sum_{v\in V(J_2)}(r-\deg_B(v))\leq\sum_{v\in V(B)}(r-\deg_B(v))=3,\]
contradiction.
\end{proof}

On the other hand, a member of $\cb_3^*(2,1,1)$ has a vertex of degree $1$, thus its neighbour is a separating vertex.

\begin{rem}
	\label{rem:br}
It follows from Definition \ref{def:b*} that if $B\in\cb_r^*(2,1,1)$, $3\leq r\leq 5$, has a separating vertex $b$, then $B-b$ has exactly two connected components, one containing the red vertex (i.e., colour $c_1$), and the other containing the blue and green vertices.
\end{rem}

The following result relates the $2$-connectivity of $A\in\ca(2,1,1)$ and $B\in\cb_r^*(2,1,1)$ with the $2$-connectivity of their product $A\t B$, where $f$ is a function that preserves colours.

\begin{lemma}
	\label{le:br}
Let $A\in\ca(2,1,1)$, $B\in\cb_r^*(2,1,1)$, $3\leq r\leq 5$, and $f$ be a function that preserves colours. Then $(a,b)$ is a separating vertex of $A\t B$ if and only if $a$ is a separating vertex of $A$ and $b$ is a separating vertex of $B$.
\end{lemma}
\begin{proof}
Let $a$ be a separating vertex of $A$, and $b$ a separating vertex of $B$. By definition of $\ca(2,1,1)$, $a$ has a neighbour $a_1$ of colour $c_1$ and a neighbour $a_2$ of colour $\neq c_1$ that lie in different connected components of $A-a$. According to Remark \ref{rem:br}, $f(a_1)$ and $f(a_2)$ are in distinct components of $B-b$, so that $b$ lies on every $f(a_1)f(a_2)$-path. Recalling Theorem \ref{thm:sep}, $A,B$ satisfy condition \ref{eq:sepi}, hence $(a,b)$ is a separating vertex of $A\t B$.

Vice versa, let $(a,b)$ be a separating vertex of $A\t B$. Then one of the conditions \ref{eq:sepi}, \ref{eq:sepii}, \ref{eq:sepiii} of Theorem \ref{thm:sep} holds. By definition of $\ca(2,1,1)$, every vertex $u$ of $A$ satisfies
\[|\{f(u') : uu'\in E(A)\}|=3,\]
hence \ref{eq:sepiii} does not hold. By Remark \ref{rem:br}, \ref{eq:sepii} does not hold. Thereby, \ref{eq:sepi} holds, thus $a$ is a separating vertex of $A$ and $b$ is a separating vertex of $B$.
\end{proof}

To prove Theorem \ref{thm:class}, we will also require a technical result specifically on necessary and sufficient conditions for the $3$-connectivity of $A\t B$ when $A\in\ca(2,1,1)$ and $B\in\cb_r^*(2,1,1)$. Recall Definition \ref{def:hash} for the subclass $\ca_\#(2,1,1)$ of $\ca(2,1,1)$.

\begin{lemma}
	\label{le:3conn}
Let $A\in\ca(2,1,1)$, $B\in\cb_r^*(2,1,1)$, and $f$ a function that preserves colours. Then $A\t B$ is $3$-connected if and only if $A\in\ca_\#(2,1,1)$ and $A$ is $2$-connected.
\end{lemma}
\begin{proof}
Let $A\t B$ be $3$-connected. By contradiction, let $a$ be a separating vertex of $A$. By definition of Sierpi\'{n}ski product, given vertices $a',a''$ in distinct components of $A-a$, then every path between a vertex in $a'B$ and a vertex in $a''B$ contains elements of $aB$. On the other hand, by Definition \ref{def:b*}, the only vertices of $aB$ adjacent to vertices outside of $aB$ are $(a,x_1)$, $(a,x_2)$, and $(a,x_3)$, where $x_i$ is the vertex in $B$ of colour $c_i$. Moreover, $(a,x_3)$ is adjacent to exactly one vertex outside of $aB$, thus $\{(a,x_1),(a,x_2)\}$ is a $2$-cut in $A\t B$, contradiction. Thereby, $A$ is $2$-connected.

Take a $2$-cut $\{a_1,a_2\}$ in $A$. By definition of Sierpi\'{n}ski product, given vertices $a',a''$ in distinct components of $A-a_1-a_2$, then every path between a vertex in $a'B$ and a vertex in $a''B$ contains elements of $a_1B$ or $a_2B$.

Suppose that $B$ has a separating vertex $b$. Recalling Remark \ref{rem:br}, $B-b$ has exactly two connected components, one containing the red vertex $x_1$ and the other containing the blue and green vertices $x_2,x_3$. Since $A\t B$ is $3$-connected, the set $\{(a_1,b),(a_2,b)\}$ is not a cut-set. On the other hand, $\{a_1,a_2\}$ is a $2$-cut in $A$. Therefore, there exists $i\in\{1,2\}$ such that $(a_i,x_1)$ is adjacent to vertices $(u_j,y_j)\in V(H_j)$ for all components $H_j$ of $A-a_1-a_2$. It follows that $A-a_1-a_2$ has exactly two components $H_1,H_2$. Similarly, in order for $\{(a_1,b),(a_2,b)\}$ to not be a cut-set, $(a_i,x_2)$ is adjacent to a vertex in $H_1$ and $(a_i,x_3)$ is adjacent to a vertex in $H_2$. Thereby, $H_1$ contains the blue vertex and one of the red vertices adjacent to $a_i$, while $H_2$ contains the green vertex and the other red vertex adjacent to $a_i$ i.e., indeed $A\in\ca_\#(2,1,1)$.

Now suppose instead that $B$ is $2$-connected. Recall that given vertices $a',a''$ in distinct components of $A-a_1-a_2$, then every path between a vertex in $a'B$ and a vertex in $a''B$ contains elements of $a_1B$ or $a_2B$. Since $A\t B$ is $3$-connected, the set
\[\{(a_1,x_i),(a_2,x_i')\}, \quad i,i'\in\{1,2,3\}\]
is not a cut-set. Hence there exists $i\in\{1,2\}$ such that $(a_i,x_1)$ is adjacent to some $(u_{1},y_{1})$ where $u_{1}\in V(H_1)$ and to some $(u_{2},y_{2})$ where $u_{2}\in V(H_2)$, with $H_1,H_2$ distinct connected components of $A-a_1-a_2$. Moreover if $(a_i,x_2)$ is adjacent to some $(v_{j},z_{j})$ where $v_j$ is a vertex in a component $H_j$ of $A-a_1-a_2$, then $(a_i,x_2)$ cannot be the only vertex in $a_iB$ adjacent to vertices with first entry in $V(H_j)$. The same statement may be made for $(a_i,x_3)$. Recalling that $(a_i,x_1),(a_i,x_2),(a_i,x_3)$ are the only vertices in $a_iB$ adjacent to vertices outside of $a_iB$, it follows that $A-a_1-a_2$ has exactly two components $H_1,H_2$, and $H_1$ contains the blue vertex and one of the red vertices adjacent to $a_i$, while $H_2$ contains the green vertex and the other red vertex adjacent to $a_i$ i.e., $A\in\ca_\#(2,1,1)$.

Vice versa, let $A$ be $2$-connected, $A\in\ca_\#(2,1,1)$, $B\in\cb_r^*(2,1,1)$, and let us show that $A\t B$ is $3$-connected. By Lemma \ref{le:br}, since $A$ is $2$-connected we already know that $A\t B$ is $2$-connected. By contradiction, let
\[\{(a_1,b_1),(a_2,b_2)\}\]
be a $2$-cut in $A\t B$. Calling $G_1,G_2,\dots,G_n$ the components of $A\t B-(a_1,b_1)-(a_2,b_2)$, we define
\[H_j=\{v\in V(A-a_1-a_2) : (v,b_1)\in V(G_j)\}, \quad 1\leq j\leq n.\]
By definition of Sierpi\'{n}ski product, $H_1,H_2,\dots,H_n$, $n\geq 2$ are the components of $A-a_1-a_2$. Hence $\{a_1,a_2\}$ is a $2$-cut in $A$, and in particular $a_1\neq a_2$ as $A$ is $2$-connected. Since $A\in\ca_\#(2,1,1)$, we deduce that $n=2$, and moreover there exists $i\in\{1,2\}$ such that $H_1$ contains the blue neighbour and one of the red neighbours of $a_i$, while $H_2$ contains the green neighbour and the other red neighbour of $a_i$. According to Remark \ref{rem:br}, if $b$ is a separating vertex of $B$, then $B-b$ has exactly two components, one containing the red vertex $x_1$, and the other containing the blue and green vertices $x_2,x_3$. Thus (whether $B$ has a separating vertex or not) we may take
\[b_1,b_2\in\{x_1,x_2,x_3\}\]
and moreover, there exists a $x_2x_3$-path $\calP$ in $B$ not containing $x_1$. Hence there exists a path in $A\t B$ between $G_1$ and $G_2$ containing $(a_i,x_1)$ but neither of $(a_i,x_2),(a_i,x_3)$, and there exists a path in $A\t B$ between $G_1$ and $G_2$ containing $(a_i,x_2),(a_i,x_3)$ but not $(a_i,x_1)$, via the sub-path
\[(a_i,x_2),\dots,(a_i,x_3),\]
where the first entries are all equal to $a_i$, and the second entries are the elements of $\calP$ taken in order. Since $\{(a_1,b_1),(a_2,b_2)\}$ is a $2$-cut in $A\t B$, it follows that $a_1=a_2$ and w.l.o.g.\ $b_1=x_1$ and $b_2=x_2$, contradicting $a_1\neq a_2$.
\end{proof}

\subsection{Proof of Theorem \ref{thm:class}}

Let $A\t B$ be an $r$-regular polyhedron, $3\leq r\leq 5$. By Proposition \ref{prop:col}, there exists $1\leq k\leq 3$ such that $A\in\ca(n_1,n_2,\dots,n_k)$ and $B\in\cb_r(n_1,n_2,\dots,n_k)$, and $f$ maps each vertex of $A$ to the vertex of $B$ of the same colour. Since $A\t B$ is $3$-connected, we have $B\in\cb_r^*(n_1,n_2,\dots,n_k)$. By Lemma \ref{le:deg3}, we know that $A,B$ are connected, planar graphs, and $\delta(A)\geq 3$. We deduce that
\[3\leq n_1+n_2+\dots+n_k\leq 5.\]
Also due to Lemma \ref{le:deg3}, $k=|\ii(f)|\geq 2$. 

Suppose that $B\simeq K_2$. Then $k=2$, so that by definition of Sierpi\'{n}ski product each vertex of $A$ has degree
\[(r-1)+(r-1),\]
hence by planarity $r=3$, thus $A\in\ca(2,2)$. 
Since $A\t B$ is $3$-connected, by Lemma \ref{le:A22conn}, $A$ is $3$-connected. We have obtained scenario 1.

Henceforth assume that $B\not\simeq K_2$. By Corollary \ref{cor:imf}, we have $k\geq 3$, so that ultimately $k=3$.


	
Let $A\t B$ be a cubic polyhedron. We write the degree sequence of $B$,
\[\sigma(B) : 3,3,\dots,3,2^y,1^z, \qquad y,z\geq 0.\]
By Proposition \ref{prop:col}, the number $y+z$ of coloured vertices in $B$ equals $k=|\ii(f)|=3$, and moreover $n_1+n_2+\dots+n_k=y+2z$, hence
\[\begin{cases}
	y+z=3
	\\
	3\leq y+2z\leq 5.
\end{cases}\]
There are three cases to inspect for $(y,z)$, namely
$(1,2)$, $(2,1)$, and $(3,0)$.

If $y=1$ and $z=2$, then $A\in\ca(2,2,1)$, contradicting Proposition \ref{prop:A32}. Let $y=3$ and $z=0$. Then
\[\sigma(B) : 3,3,\dots,3,2,2,2,\]
and $A$ is $3$-regular, with colours assigned to the vertices so that each has neighbours coloured red, blue, green. That is to say, $A\in\ca(1,1,1)$. By Lemma \ref{le:A111}, $A$ is $2$-connected. Let us prove $3$-connectivity for $A$. By contradiction, let $\{a_1,a_2\}$ be a $2$-cut in $A$. Since $a_1,a_2$ are of degree $3$, then there exist edges $a_1a_3$ and $a_2a_4$ such that
\[A-a_1a_3-a_2a_4\]
is disconnected. Let $c_i$ be the colour of $a_3$ and $c_j$ the colour of $a_4$. Call $b_i$ the vertex of $B$ of colour $c_i$, and $b_j$ the vertex of $B$ of colour $c_j$. Then
\[\{(a_1,b_i),(a_2,b_j)\}\]
is a $2$-cut in $A\t B$, contradiction. We have obtained scenario 2.

Let $y=2$ and $z=1$. Then
\[\sigma(B) : 3,3,\dots,3,2,2,1,\]
and $A$ is $4$-regular, with colours assigned to the vertices so that each has neighbours coloured red, red, blue, green in this cyclic order. That is to say, $A\in\ca(2,1,1)$. Applying Lemma \ref{le:br}, since $B$ has a vertex of degree $1$ and $A\t B$ does not have a separating vertex, we deduce that $A$ is $2$-connected. Furthermore, by Lemma \ref{le:3conn}, $A\in\ca_\#(2,1,1)$. We have obtained scenario 3.

Now let $A\t B$ be a quartic polyhedron. We have $A\in\ca(n_1,n_2,n_3)$, and $B\in\cb_4(n_1,n_2,n_3)$. We write the degree sequence of $B$,
\[\sigma(B) : 4,4,\dots,4,3^x,2^y,1^z, \qquad x,y,z\geq 0.\]
By Proposition \ref{prop:col}, the number $x+y+z$ equals $k=|\ii(f)|=3$, and moreover $n_1+n_2+n_3=x+2y+3z$, hence
\[\begin{cases}
	x+y+z=3
	\\
	3\leq x+2y+3z\leq 5.
\end{cases}\]
We deduce that $z\leq 1$. Further, by the handshaking lemma, $x+z$ is even. Consolidating these conditions, we see that $z=1$ implies
\[2=x+y\leq x+2y\leq 2,\]
thus $y=0$ and $x=2$, but then $x+z$ is odd, contradiction. Thereby, $z=0$. It follows that $x+y=3$, $x+2y\leq 5$, and $x$ is even. If $x=0$ then $y=3$ hence $x+2y=6$, contradiction.

Hence $x=2$, thus $y=1$. It follows that
\[\sigma(B) : 4,4,\dots,4,3,3,2.\]
We also obtain $n_1+n_2+n_3=4$, hence $A$ is $4$-regular. In fact, $n_1=4-2=2$, and $n_2=n_3=4-3=1$, thus every $a\in V(A)$ has neighbours coloured red, red, blue, green in this cyclic order. That is to say, $A\in\ca(2,1,1)$. By Lemma \ref{le:3conn}, in fact $A\in\ca_\#(2,1,1)$ and $A$ is $2$-connected. We have obtained scenario 4.

Now let $A\t B$ be a quintic polyhedron. By Proposition \ref{prop:col}, we have $A\in\ca(n_1,n_2,n_3)$, and $B\in\cb_5(n_1,n_2,n_3)$. We write the degree sequence of $B$,
\[\sigma(B) : 5,5,\dots,5,4^w,3^x,2^y,1^z, \qquad w,x,y,z\geq 0.\]
By Proposition \ref{prop:col}, the number $w+x+y+z$ of coloured vertices in $B$ equals $|\ii(f)|=3$, and moreover $n_1+n_2+n_3=w+2x+3y+4z$, hence
\[\begin{cases}
	w+x+y+z=3
	\\
	3\leq w+2x+3y+4z\leq 5.
\end{cases}\]
We deduce that $z\leq 1$. If $z=1$, then $w+2x+3y+4z\geq 2+4=6$, impossible. It follows that $z=0$. We are left with
\[\begin{cases}
	w+x+y=3
	\\
	3\leq w+2x+3y\leq 5.
\end{cases}\]
There are four solutions for $(w,x,y)$, namely
\[(1,2,0),\ (2,0,1),\ (2,1,0),\ (3,0,0).\]

If $w=1$, $x=2$ and $y=0$, then $A\in\ca(2,2,1)$. If $w=2$, $x=0$ and $y=1$, then $A\in\ca(3,1,1)$. However these classes of graphs are both empty due to Proposition \ref{prop:A32}. Let $w=3$, and $x=y=0$. Then
\[\sigma(B) : 5,5,\dots,5,4,4,4,\]
and $A$ is $3$-regular, with colours assigned to the vertices so that each has neighbours coloured red, blue, green. That is to say, $A\in\ca(1,1,1)$. We have obtained scenario 5 (the $3$-connectivity of $A$ is proven as in scenario 2).

Let $w=2$, $x=1$, and $y=0$. Then
\[\sigma(B) : 5,5,\dots,5,4,4,3,\]
and $A$ is $4$-regular, with colours assigned to the vertices so that each has neighbours coloured red, red, blue, green in this cyclic order. That is to say, $A\in\ca(2,1,1)$. By Lemma \ref{le:3conn}, we deduce that $A\in\ca_\#(2,1,1)$ and $A$ is $2$-connected. We have obtained scenario 6.

Vice versa, let us show that in each of the six scenarios, if $A,B$ have the properties described in Table \ref{t:1}, and $f$ preserves colours, then $A\t B$ is a regular polyhedron. Regularity and planarity of $A\t B$ follow from Proposition \ref{prop:col}. It remains to prove $3$-connectivity of the product in each scenario.

In scenario 1, since $A$ is $3$-connected, by Lemma \ref{le:A22conn} $A\t B$ is also $3$-connected. In scenarios 3, 4, and 6, by Lemma \ref{le:3conn} $A\t B$ is $3$-connected.

In scenarios 2 and 5, by contradiction, let $(a,b)$ be a separating vertex of $A\t B$. By Theorem \ref{thm:sep}, either $a$ is a separating vertex of $A$, contradicting Lemma \ref{le:A111}, or $b$ is a separating vertex of $B$, contradicting Lemma \ref{le:B111}, or every neighbour of $a$ has the same image under $f$ (i.e., they are of the same colour), impossible as $A\in\ca(1,1,1)$. Thus $A\t B$ is $2$-connected. Now by contradiction let
\[
\{(a_1,b_1),(a_2,b_2)\}
\]
be a minimal $2$-cut in $A\t B$. By Proposition \ref{prop:kc}, either $\{a_1,a_2\}$ is a $2$-cut in $A$, contradicting the assumption that $A$ is $3$-connected, or $a_1=a_2=a$ and $\{b_1,b_2\}$ is a $2$-cut in $B$, or $a_1=a_2=a$ and every neighbour of $a$ has one of two colours, impossible as $A\in\ca(1,1,1)$. Analysing the proof of Proposition \ref{prop:kc}, we deduce that there exist $x,y\in V(B)$ such that every $xy$-path in $B$ contains one of $b_1,b_2$, and moreover every $(a,x)(a,y)$-path in $A\t B$ contains one of $\{(a,b_1),(a,b_2)\}$. Let $J_x,J_y$ be the components of $B-b_1-b_2$ containing $x,y$ respectively. By \eqref{eq:hsum}, $J_x$ contains a vertex $u$ of degree $2$ in $B$, and $J_y$ contains a vertex $v\neq u$ of degree $2$ in $B$. We find a path $P_1$ in $aJ_x$ from $(a,x)$ to $(a,u)$, a path $P_2$ in $aJ_y$ from $(a,v)$ to $(a,y)$. Let $a_u,a_v$ be the neighbours of $a$ of same colour as $u,v$ respectively. Since $A-a$ is connected, we can find in this graph a path between $a_u$ and $a_v$. by Lemma \ref{le:path}, we can find a path $P$ in $A\t B$ from $(a_u,f(a))$ to $(a_v,f(a))$ that does not contain any vertex from $aB$. Combining the paths and edges
\[P_1,(a,u)(a_u,f(a)),P,(a_v,f(a))(a,v),P_2,\]
we find a $(a,x)(a,y)$-path in $A\t B$ containing neither of $(a,b_1),(a,b_2)$, contradiction. Hence $A\t B$ is indeed $3$-connected. The proof of Theorem \ref{thm:class} is complete.

\section{Families of polyhedral Sierpi\'{n}ski products}
\label{sec:fam}

\subsection{$K_4\t B$}
In this section, we characterise all polyhedra of the form $K_4\t B$, without assuming that $B$ is a polyhedron as we did in Theorem \ref{thm:siepol}, and without assuming that the product is a regular graph, as we did in Theorem \ref{thm:class}. These products are notable as $K_4$ is the smallest polyhedron.

We write $V(K_4)=\{a_1,a_2,a_3,a_4\}$.
\begin{prop}
	\label{prop:K4B}
Let $B\not\simeq K_4$. Then $K_4\t B$ is a polyhedron if and only if the following conditions are all fulfilled:
\begin{itemize}
\item
$f$ is injective;
\item
$B$ is planar and $2$-connected;
\item
if $\{b',b''\}$ is a $2$-cut in $B$, then the graph $B-b'-b''$ contains exactly two connected components, and each of them contains exactly two of $f(a_1),f(a_2),f(a_3),f(a_4)$;
\item
there exists a planar embedding of $B$ such that $f(a_1),f(a_2),f(a_3),f(a_4)$ all lie on the same region.
\end{itemize}
\end{prop}
\begin{proof}
For the first part of this proof, we will assume that $A\t B$ is a polyhedron. 
Let us show that $f$ is injective. Since $A\t B$ is $3$-connected, there exist at least three vertices in $a_1B$ adjacent to vertices outside of $a_1B$. Therefore, $f(a_2),f(a_3),f(a_4)$ are distinct. Repeating the same reasoning for $a_2B$, $a_3B$, and $a_4B$, we conclude that \[f(a_1),f(a_2),f(a_3),f(a_4)\]
are all distinct i.e., $f$ is injective.

Since $A\t B$ is a polyhedron, then by Proposition \ref{le:deg3} $B$ is a connected, planar graph with at least three vertices. Let us show that $B$ is $2$-connected. By contradiction, let $b$ be a separating vertex of $B$. Then $(a_1,b)$ is a separating vertex of $a_1B$, thus the graph
\begin{equation}
	\label{eq:k4r}
a_1B-(a_1,b)
\end{equation}
has at least two connected components. The only vertices in $a_1B$ adjacent to vertices outside of $a_1B$ are
\begin{equation}
	\label{eq:k43v}
(a_1f(a_2)),(a_1f(a_3)),(a_1f(a_4)).
\end{equation}
Hence there exists a connected component of \eqref{eq:k4r} containing at most one of \eqref{eq:k43v}. If this component contains none of \eqref{eq:k43v}, then
\[A\t B-(a_1,b)\]
is disconnected, contradiction. If instead there exists a component $J$ of \eqref{eq:k4r} containing $(a_1f(a_2))$ but neither $(a_1f(a_3))$ nor $(a_1f(a_4))$, then either
\[A\t B-(a_1,b)-(a_1,f(a_2))\]
is disconnected -- contradiction, or $J$ is just the vertex $(a_1,f(a_2))$. In the latter case, the only neighbours of $(a_1,f(a_2))$ in $A\t B$ are $(a_1,b)$ and $(a_2,f(a_1))$, contradiction. Hence $B$ is indeed $2$-connected.

Now let $A\t B$ be a polyhedron, and $\{b',b''\}$ a $2$-cut in $B$. Then for each $i=1,2,3,4$, the graph $a_iB$ has the $2$-cut $\{a_ib',a_ib''\}$. Let us assume by contradiction that a connected component of $B-b'-b''$ contains at most one of $f(a_1),f(a_2),f(a_3),f(a_4)$, say $f(a_j)$. Then $A\t B$ has the $2$-cut $\{a_jb',a_jb''\}$, contradiction. Thereby, each connected component of $B-b'-b''$ contains at least two of $f(a_1),f(a_2),f(a_3),f(a_4)$. It follows that $B-b'-b''$ has exactly two connected components, and each of them contains exactly two of $f(a_1),f(a_2),f(a_3),f(a_4)$.

We have already checked that $B$ is a planar, $2$-connected graph. We will now show that there exists a planar embedding of $B$ such that $f(a_1),f(a_2),f(a_3),f(a_4)$ all lie on the same region. We sketch $a_1B$ in the plane such that $(a_1,f(a_2))$ belongs to the boundary of the exterior region $\calR$. By contradiction, suppose that $(a_1,f(a_3))$ does not belong to $\calR$. In the sketch of $A\t B$, by planarity $a_2B$ will be outside of $\calR$ and $a_3B$ inside. Then the edge
\[(a_2,f(a_3))(a_3,f(a_2))\]
crosses the boundary of $\calR$, contradiction. Permuting the indices we deduce that there exists a planar embedding of $B$ such that given any three of
\begin{equation}
	\label{eq:1234}
f(a_1),f(a_2),f(a_3),f(a_4)
\end{equation}
they belong to the same region. Suppose by contradiction that no region contains all four. Then w.l.o.g.\ we may start the sketch of $B$ by drawing a tetrahedron of vertices \eqref{eq:1234} (and subsequently subdivide the edges of this tetrahedron if necessary). Since $B\not\simeq K_4$, the sketch is not yet complete. Any further vertices must lie on the existing edges, not inside the existing faces, otherwise there would exist three of \eqref{eq:1234} not lying on the same region. Hence w.l.o.g.\ we insert $b_5$ along the edge $f(a_1)f(a_2)$. We must also add at least one further edge, otherwise $b_5$ would be of degree $2$ in $B$, thus $(a_1,b_5)$ would be of degree $2$ in $A\t B$, impossible. Any new path we draw starting from $b_5$ must end somewhere along the edge $f(a_1)f(a_2)$, else there would exist three of \eqref{eq:1234} not lying on the same region. However, this means that
\[(a_1,f(a_1)),(a_1,f(a_2))\] 
is a $2$-cut in $A\t B$, contradiction.

Vice versa, let $B$ be a $2$-connected plane graph, and $f$ a function such that \eqref{eq:1234} are distinct vertices on the boundary of the external face of $B$. We assume moreover that any $2$-cut in $B$ results in exactly two components, each containing two of \eqref{eq:1234}. Our aim is to show that $K_4\t B$ is a polyhedron. To check planarity, we sketch as in Figure \ref{fig:K4B} and use the fact that $B$ is planar.
\begin{figure}[ht]
	\centering
	\includegraphics[width=6.5cm]{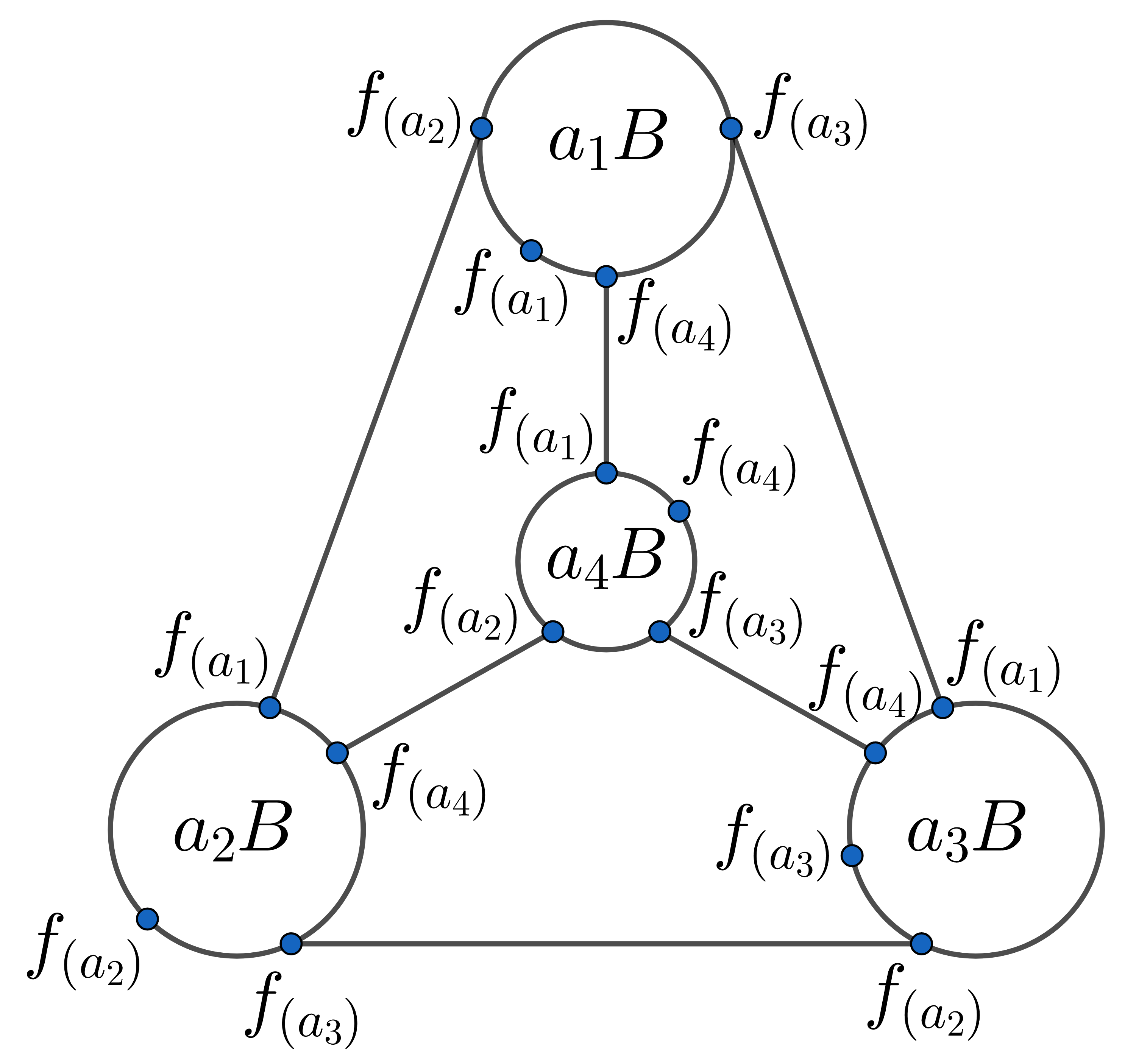}
	\caption{The labels are the second entries of the vertices. $f(a_1),f(a_2),f(a_3),f(a_4)$ appear around a region of $B$ in this order.}
	\label{fig:K4B}
\end{figure}

It remains to check $3$-connectivity. First, given any two of $a_1B,a_2B,a_3B,a_4B$, there is an edge connecting them. Second, for each $i\in\{1,2,3,4\}$, any $2$-cut of $a_iB$ results in exactly two components, each containing two of
\[(a_i,f(a_1)),(a_i,f(a_2)),(a_i,f(a_3)),(a_i,f(a_4)),\]
hence the $2$-cut of $a_iB$ is not a $2$-cut of $A\t B$. This completes the proof.
\end{proof}


\subsection{$A\t K_2$}
The following result generalises scenario 1 of Theorem \ref{thm:class}. We drop the condition of regularity of the product, and characterise all graphs $A$ and maps $f$ such that $A\t K_2$ is a polyhedron. 
\begin{prop}
	\label{prop:rb}
	Let $A$ be a graph. There exists a function
	\[f:V(A)\to V(K_2)\]
	such that
	\[A\t K_2\]
	is a polyhedron if and only if $A$ is a polyhedron satisfying \[\delta(\alpha)\geq 4\]
	and we may assign to each element of $V(A)$ one of two colours red and blue, in such a way that for every vertex, when we consider the cyclic orientation of its neighbours around it, those of one colour are all consecutive, and moreover there are at least two neighbours of each colour.
\end{prop}
\begin{proof}
$\Rightarrow$. Let $f$ be such that $A\t K_2$ is a polyhedron. By Lemma \ref{le:deg3}, $A$ is a planar graph satisfying $\delta(A)\geq 4$. We write $V(K_2)=\{x,y\}$. For each $a\in V(A)$, we assign the colour red if $f(a)=x$, and the colour blue if $f(a)=y$. 
In the planar sketch of $A\t K_2$, for each $a\in V(A)$ the vertex $(a,x)$ is adjacent to $(a,y)$ and to
\[(a_1,z),(a_2,z),\dots,(a_i,z)\]
in this cyclic order around $(a,x)$, where $i\geq 2$, $a_1,a_2,\dots,a_i$ are red vertices of $A$, and $z=x$ if $a$ is red, while $z=y$ if $a$ is blue. Similarly, $(a,y)$ is adjacent to $(a,x)$ and to
\[(a_{i+1},z),(a_{i+2},z),\dots,(a_{i+j},z)\]
in this cyclic order around $(a,y)$, where $j\geq 2$ and $a_{i+1},a_{i+2},\dots,a_{i+j}$ are blue vertices of $A$.
\\
We now contract all edges of $A\t K_2$ of type
\[(a,x)(a,y), \quad a\in V(A).\]
The result is a plane graph isomorphic to $A$. Since $A\t B$ is $3$-connected, then so is $A$. By construction, for each $a\in V(A)$, the cyclic order of its neighbours around $a$ is either
\[a_1,a_2,\dots,a_i,a_{i+1},a_{i+2},\dots,a_{i+j}\]
or
\[a_i,a_{i-1},\dots,a_1,a_{i+1},a_{i+2},\dots,a_{i+j}.\]
Hence those of same colour are consecutive in the cyclic order around $a$, as claimed. 

$\Leftarrow$. Writing $V(K_2)=\{x,y\}$, we define for every $a\in V(A)$
\[f(a)=\begin{cases}
	x & a \text{ is red};
	\\y & a \text{ is blue}.
\end{cases}\]
The assumptions of \cite[Theorem 2.13]{kpzz19} are all satisfied, hence $A\t K_2$ is a planar graph. It remains to check its $3$-connectivity. By contradiction, there exist
\[(a',z'), (a'',z'')\in V(A\t K_2)\]
with $z',z''\in\{x,y\}$ such that
\[A\t K_2-(a',z')-(a'',z'')\]
is disconnected. By construction, if in $A\t K_2$ we contract all edges $(a,x)(a,y)$ for every $a\in V(A)$, we obtain a graph isomorphic to $A$. Thus $A-a'-a''$ is disconnected, contradiction.	
\end{proof}

\subsection{$A\t K_3$}
In this section, we revisit scenario 2 of Theorem \ref{thm:class}, taking $B=K_3$, the smallest graph in $\cb_3^*(1,1,1)$. Given a polyhedron $G$, the operation of truncation of each vertex may be described as follows. For each $u\in V(G)$, we consider its neighbours $v_1,v_2,\dots,v_n$ in cyclic order around $u$, and perform
\[G-u+w_1+w_2+\dots+w_n+w_1v_1+w_2v_2+\dots+w_nv_n+w_1w_2+w_2w_3+\dots+w_nw_1.\]
The obtained graph is a polyhedron called the truncation of $G$.
\begin{rem}
	\label{rem:trun}
	Let $A\in\ca(1,1,1)$ be a polyhedron, and take a copy of $K_3$ where the three vertices have the colours $c_1,c_2,c_3$. If $f:V(A)\to V(K_3)$ is a function that preserves colours, then the polyhedron $A\t K_3$ is isomorphic to the truncation of $A$.
\end{rem}
For instance, the truncation of the triangular prism in Figure \ref{fig:A111bis} is depicted in Figure \ref{fig:trun}.
\begin{figure}[ht]
	\centering
	\begin{subfigure}{0.4\textwidth}
		\centering
		\includegraphics[width=3.cm]{A111.png}
		\caption{$A\in\ca(1,1,1)$.}
		\label{fig:A111bis}
	\end{subfigure}
	\begin{subfigure}{0.58\textwidth}
		\centering
		\includegraphics[width=3.75cm]{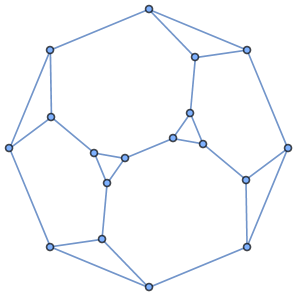}
		\caption{$A\t K_3$ is the truncation of $A$.}
		\label{fig:trun}
	\end{subfigure}
	\caption{Remark \ref{rem:trun}.}
	\label{fig:trunc}
\end{figure}

\appendix
\clearpage
\section{Constructions for $\cb_r^*(1,1,1)$ and $\cb_r^*(2,1,1)$.}
\label{sec:appa}

We start with a complete characterisation of the classes $\cb_3^*(1,1,1)$ and $\cb_3^*(2,1,1)$.
\begin{prop}
	We denote by $\cg_3(1)$ the class of graphs obtained by deleting any vertex from a cubic polyhedron, and $\cg_3(2)$ the class of graphs obtained by deleting any two adjacent edges from a cubic polyhedron. Then
	\[\cb_3^*(1,1,1)=\cg_3(1)
	\qquad\text{and}\qquad
	\cb_3^*(2,1,1)=\cg_3(2).\]
\end{prop}
\begin{proof}
	Let $B\in\cb_3^*(1,1,1)$. By definition, there exists a planar embedding of $B$ such that the three vertices of degree $2$ all lie on one region. We add a vertex $v$ inside of this region, and edges between $v$ and the three vertices of degree $2$ in $B$, obtaining a $3$-regular, plane graph $G$. By contradiction, let $\{b_1,b_2\}$ be a $2$-cut in $G$. By Lemma \ref{le:B111}, $B$ is $2$-connected, thus $v\not\in\{b_1,b_2\}$, so that $\{b_1,b_2\}$ is also a $2$-cut in $B$. By assumption, every component of $B-b_1-b_2$ contains at least one of the vertices of degree $2$. Since these are all adjacent to $v$ in $G$, then $\{b_1,b_2\}$ cannot be a $2$-cut in $G$, contradiction.
	
	Vice versa, let $G$ be a cubic polyhedron, $v\in V(G)$, and $B:=G-v$. Then $B$ is a $2$-connected graph of degree sequence
	\[3,3,\dots,3,2,2,2,\]
	that has a planar embedding such that the vertices of degree $2$ all lie on one region $\calR$. If $\{b_1,b_2\}$ is a $2$-cut in $B$, then since $G$ is $3$-connected, both of $b_1,b_2$ lie on $\calR$, and both components of $B-b_1-b_2$ contain at least one vertex of degree $2$ in $B$. Thus $B\in\cb_3^*(1,1,1)$.
	
	Let $B\in\cb_3^*(2,1,1)$. Call $v$ the vertex of degree $1$ and $u,w$ those of degree $2$. By the handshaking lemma, there is a least one vertex of degree $3$. We claim that the neighbour of $v$ has degree $3$. By contradiction, if it has degree $2$, then we delete $u,w$ from $B$, resulting in a component containing only vertices that have degree $3$ in $B$, contradicting \eqref{eq:hsum}. We consider the new graph
	\[G:=B+vu+vw,\]
	that is $3$-regular and planar by construction. By contradiction, let $\{b_1,b_2\}$ be a $2$-cut in $G$. By assumption, every component of $B-b_1-b_2$ contains at least one of $v,u,w$. Since $vu,vw\in E(G)$, then $\{b_1,b_2\}$ cannot be a $2$-cut in $G$, contradiction. It is worth noting that $B-v\in\cb_3^*(1,1,1)$.
	
	Vice versa, let $G$ be a cubic polyhedron, $v\in V(G)$, and $u,w$ two neighbours of $v$. We consider the new graph
	\[B:=G-vu-vw.\]
	Then $B$ is a connected, planar graph of degree sequence
	\[3,3,\dots,3,2,2,1,\]
	that has a planar embedding such that the vertices of degree less than $3$ all lie on one region $\calR$. If $\{b_1,b_2\}$ is a $2$-cut in $B$, then since $G$ is $3$-connected, both of $b_1,b_2$ lie on $\calR$, and both components of $B-b_1-b_2$ contain at least one vertex of degree less than $3$ in $B$. Thus $B\in\cb_3^*(2,1,1)$.
\end{proof}

For instance, $K_3\in\cb_3^*(1,1,1)$ is obtained from the tetrahedron by removing a vertex, and the member of $\cb_3^*(2,1,1)$ with edges $\{b_1b_4,b_2b_4,b_3b_4,b_2b_3\}$ is obtained from the tetrahedron by removing two adjacent edges.

More generally, to construct elements of $\cb_r^*(2,1,1)$, $3\leq r\leq 5$, one may start with an $r$-regular polyhedron, choose any face and delete two adjacent edges on it. The class $\cb_4^*(1,1,1)$ is of course empty due to the handshaking lemma.

Lastly, to obtain members of $\cb_5^*(1,1,1)$, one may proceed as follows. For every integer $h\geq 0$ one starts with the $8+40h$-gonal antiprism. In the internal $8+40h$-gon $\calR$, we add $2+10h$ new vertices,
\[u_1,u_2,\dots,u_{2+10h}\]
each adjacent to four consecutive vertices of $\calR$, and extra edges
\[u_1u_2,u_3u_4,\dots,u_{1+10h}u_{2+10h}.\]
In the external $8+40h$-gon $\calS$, we add $1+8h$ new vertices, each adjacent to five consecutive vertices of $\calS$, leaving $3$ vertices of degree $4$ on the same face of the resulting polyhedron. The case $h=0$ is depicted in Figure \ref{fig:B5111}.
\begin{figure}[ht]
	\centering
	\includegraphics[width=4.5cm]{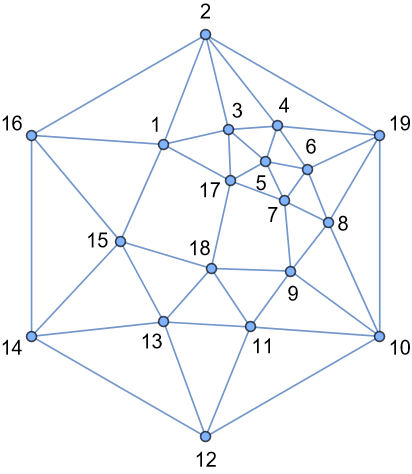}
	\caption{A $3$-connected member of $\cb_5^*(1,1,1)$. Vertices of $\calR$ are the odd numbers up to $16$, and vertices of $\calS$ are the even numbers up to $16$. Adding the vertices $17,18,19$ as described leaves $12,14,16$ of degree $4$.}
	\label{fig:B5111}
\end{figure}

\clearpage
\section{Table of regular, connected, planar Sierpi\'{n}ski products}
\label{sec:appb}

We wish to classify the regular, connected, planar Sierpi\'{n}ski products. We apply Proposition \ref{prop:col}, and then proceed as in the proof of Theorem \ref{thm:class}, without assuming higher connectivity for the product.

\begin{prop}
	\label{prop:2345}
	Let $A,B$ be non-trivial graphs and $f$ a function. Then $A\t B$ is $r$-regular, connected, and planar if and only if $A\in\ca(n_1,n_2,\dots,n_k)$, $B\in\cb_r(n_1,n_2,\dots,n_k)$, and $f$ preserves colours, where $(n_1,n_2,\dots,n_k)$ takes one of the values listed in Table \ref{t:2}.
		\begin{table}[h!]
		\centering
		$\begin{array}{|l|l|l|}
			\hline r&(n_1,n_2,\dots,n_k)&\text{case}\\
			\hline 2&(1,1)&1\\
			\hline 3&(1), A=K_2&2\\
			\hline 3&(1,1)&3\\
			\hline 3&(2)&4\\
			\hline 3&(1,1,1)&5\\
			\hline 3&(2,1)&6\\
			\hline 3&(2,1,1)&7\\
			\hline 3&(2,2)&8\\
			\hline 4&(1,1)&9\\
			\hline 4&(2)&10\\
			\hline 4&(2,1,1)&11\\
			\hline 4&(2,2)&12\\
			\hline 4&(3,1)&13\\
			\hline 5&(1), A=K_2&14\\
			\hline 5&(1,1)&15\\
			\hline 5&(2)&16\\
			\hline 5&(1,1,1)&17\\
			\hline 5&(2,1)&18\\
			\hline 5&(3)&19\\
			\hline 5&(2,1,1)&20\\
			\hline 5&(2,2)&21\\
			\hline 5&(3,1)&22\\
			\hline 5&(4)&23\\
			\hline 5&(4,1)&24\\
			\hline
		\end{array}$
		\caption{The cases where $A\t B$ is $r$-regular, connected, and planar.}
		\label{t:2}
	\end{table}
\end{prop}


\clearpage
\addcontentsline{toc}{section}{References}
\bibliographystyle{abbrv}
\bibliography{allgra}
\end{document}